\documentclass[preprint]{imsart}

\RequirePackage[OT1]{fontenc}
\RequirePackage{amsthm,amsmath,booktabs}
\RequirePackage[authoryear]{natbib}
\RequirePackage[colorlinks,citecolor=blue,urlcolor=blue]{hyperref}

\usepackage{amsthm,amsmath,natbib,amssymb,enumerate}
\usepackage{times}
\usepackage{graphicx} % more modern
\usepackage{subcaption} 

\usepackage{natbib}

\usepackage{algorithm}
\usepackage{algorithmic}

\newcommand{\bL}{\bar{\ell}}
\newcommand{\bP}{\overline{P}}
\newcommand{\bG}{\bar{G}}
\newcommand{\tC}{\widetilde{C}}

\newcommand{\hsigma}{\widehat{\sigma}}
\newcommand{\CV}{\text{CV}}
\newcommand{\select}{\text{select}}
\newcommand{\inter}{\text{inter}}

\newcommand{\Ee}{\mathbb{E}}
\newcommand{\Pp}{\mathbb{P}}
\newcommand{\Q}{\mathbb{Q}}
\newcommand{\mQ}{\widehat{\cal Q}}
\newcommand{\M}{\mathbb{M}}

\newcommand{\hbeta}{\hat{\beta}}

\newcommand{\tP}{\widetilde{P}}

\newcommand{\law}{\mathcal{L}}

\newcommand{\cH}{\mathcal{H}}
\newcommand{\cD}{\mathcal{D}}

\newcommand{\cQ}{\mathcal{Q}}

\newcommand{\bX}{X}

\newcommand{\bbeta}{\bar{\beta}}

\newcommand{\sF}{\mathcal{F}}

\newcommand{\selectionevent}{K}
\newcommand{\paramset}{\mathbf{\Theta}}

\def\mrm#1{\mathrm{#1}}

\def\integers{\mathbb{Z}} % Integer symbol
\def\norm#1{\left\|{#1}\right\|} % A norm with 1 argument
\providecommand{\argmin}{\mathop\mathrm{arg min}}

\providecommand{\diag}{\mathop\mathrm{diag}}

\providecommand{\sign}{\mathop\mathrm{sign}}

\def\supp#1{\mathrm{supp}({#1})}
\def\E{\mathbb{E}} % Expectation symbol
\def\Earg#1{\E\left[{#1}\right]}
\def\Esubarg#1#2{\E_{#1}\left[{#2}\right]}
\def\P{\mathbb{P}} % Probability symbol
\def\Parg#1{\P\left({#1}\right)}
\def\Psubarg#1#2{\P_{#1}\left[{#2}\right]}
\def\Cov{\mrm{Cov}} % Covariance symbol
\def\Covarg#1{\Cov\left[{#1}\right]}

\def\Var{\mrm{Var}} % Covariance symbol
\def\Vararg#1{\Var\left[{#1}\right]}

\def\abv#1{\left|#1\right|}
\newcommand{\iid}{\stackrel{\mathrm{iid}}{\sim}}

\newcommand{\modelweight}{\mathbb{W}}
\newcommand{\F}{\mathbb{F}}

\newcommand{\real}{\mathbb{R}}

\newtheorem{theorem}{Theorem}

\newtheorem{lemma}[theorem]{Lemma}

\newtheorem{definition}[theorem]{Definition}

\newtheorem{remark}[theorem]{Remark}
\newtheorem{example}{Example}
\def\Bin{\textnormal{Bin}}
\def\Unif{\textnormal{Unif}}

\def\Exp{\textnormal{Exp}}
\def\Logistic{\textnormal{Logistic}}

\begin{document} 
\begin{frontmatter}

\title{Selective inference with a randomized response}
\runtitle{Selective inference with a randomized response}

\begin{aug}
\author{\fnms{Xiaoying} \snm{Tian}\corref{}\ead[label=e1]{xtian@stanford.edu},  
\fnms{Jonathan} \snm{Taylor}\ead[label=e2]{jonathan.taylor@stanford.edu}\thanksref{t1}
%\fnms{Robert} \snm{Tibshirani}\ead[label=e4]{tibs@stanford.edu}
}
\runauthor{Tian and Taylor}

% Doesn't show up at all with no specification of aos, aop, etc.,  style
\affiliation{$^1$Stanford University}

\address{Department of Statistics\\ Stanford University 
\\ Sequoia Hall \\ Stanford, CA 94305, USA \\ \printead{e1} \\
\printead*{e2} } % \\ \printead*{e4}}

\thankstext{t1}{Supported in part by NSF grant DMS 1208857 and
AFOSR grant 113039.}
\end{aug}

\begin{abstract} 
Inspired by sample splitting and the reusable holdout introduced in the field
of differential privacy, we consider selective inference with a randomized
response. We discuss two major advantages of using a randomized response for
model selection. First, the selectively valid tests are more powerful after
randomized selection. Second, it allows consistent estimation and weak
convergence of selective inference procedures.  Under independent sampling, we
prove a selective (or privatized) central limit theorem that transfers
procedures valid under asymptotic normality without selection to their
corresponding selective counterparts. This allows selective inference in
nonparametric settings. Finally, we propose a framework of inference after
combining multiple randomized selection procedures.  We focus on the classical
asymptotic setting, leaving the interesting high-dimensional asymptotic
questions for future work.
%Using a randomized response can ensure that the
%leftover information of \cite{optimal_inference} is bounded below
%ensuring that selective intervals are better behaved than without randomization.
%Under independent sampling, we prove a selective (or privatized) central limit theorem
%that transfers procedures valid under asymptotic normality without selection to 
%their corresponding selective counterparts. This allows selective inference in the nonparametric
%settings. Finally, we describe 
%a method for selective inference following cross-validation using slightly more randomization
%than the split into groups of standard cross-validation.
%We focus on the classical asymptotic setting, leaving the interesting
%high-dimensional asymptotic questions for future work.
\end{abstract} 

\begin{keyword}[class=AMS]
\kwd[Primary ]{62M40}
\kwd[; secondary ]{62J05}
\end{keyword}

\begin{keyword}
\kwd{selective inference}
\kwd{nonparametric}
\kwd{differential privacy}
\end{keyword}

\end{frontmatter}

\section{Introduction}
\label{introduction}

\cite{EDA} promoted the use of {\em exploratory data analysis} to examine
the data and possibly formulate hypotheses for further investigation. Nowadays,
many statistical learning methods allow us to perform these exploratory
data analyses, based on which we can posit a model on the data generating distribution. 
Since this model is not given a priori, classical statistical inference will not
provide valid tests that control the Type-I errors. 

{\em Selective inference} seeks
to address this problem, see \cite{exact_lasso, covtest, exact_screening, optimal_inference}.
Loosely speaking, there are two stages in selective inference. The first is the
{\em selection} stage that explores the data and formulates a plausible model
for the data distribution. Then we enter the {\em inference} stage that seeks
to provide valid inference under the selected model which is proposed after inspecting
the data.
Inference under different models have been studied, notably the Gaussian families
\cite{exact_lasso, tian2015selective, exact_screening} as well as other exponential
families \cite{optimal_inference}. 

In this work, we consider selective inference in a general setting that
include nonparametric settings. In addition, we introduced the use of
{\em randomized response} in model selection. A most common example of
randomized model selection is probably the practice of data splitting.
Assuming independent sampling, we can divide the data into two subsets, using
the first for model selection and the second subset for inference.
Though not emphasized, this split is often {\em random}. Hence, data splitting
can be thought of as a special case of randomized model selection.  To motivate
the use of randomized selection and introduce the inference problem that
ensues, we consider the following example.

\subsection{A first example}
\label{sec:file:drawer}

Publication bias, (also called the ``file drawer effect'' by \cite{file_drawer}) is a bias
introduced to scientific literature by failure to report negative or non-confirmatory results.
We formulate the problem in the simple example below.

\begin{example}[File drawer problem]
\label{example1}

Let 
$$\bar{X}_n = \frac{1}{n} \sum_{i=1}^n X_{i,n}$$ be the sample mean of a sample of $n$ $\mrm{iid}$ draws
from $\F_n$ in a standard triangular array. We set $\mu_n = \E_{\F_n}[X_{1,n}]$
and assume $\E_{\F_n}[(X_{1,n}-\mu_n)^2]=1$.

Suppose that we are interested in discovering positive effects
and would only report the sample mean if it survives the file drawer effect, i.e. 
\begin{equation}
\label{eq:file:drawer}
n^{1/2} \bar{X}_n > 2.
\end{equation}
Then what is the ``correct'' $p$-value to
report for an observation $\bar{X}_{n, obs}$ that exceeds the threshold? 
\end{example}

If we have Gaussian family, namely $\F_n = N(\mu_n, 1)$, then the
distribution of $\bar{X}_n$ surviving the file drawer effect \eqref{eq:file:drawer} is
a truncated Gaussian distribution. We also call this distribution
the {\em selective distribution}. Formally, its survival function is
$$
\begin{aligned}
P(t) &= \P\left( \bar{X}_n > t | n^{1/2} \bar{X}_n > 2 \right), \quad
\bar{X}_n \sim N\left(\mu_n, \frac{1}{n}\right) \\
&= \frac{1 - \Phi\left(n^{1/2}(t-\mu_n)\right)}{1 - \Phi(2-n^{1/2}\mu_n)}
\end{aligned}
$$
where $\Phi$ is the CDF of an $N(0,1)$ random variable. Therefore, we get a pivotal
quantity 
\begin{equation}
\label{eq:pivot:univariate}
\begin{aligned}
&P(\bar{X}_{n,obs}) = \frac{1 - \Phi\left(n^{1/2}(\bar{X}_{n,obs}-\mu_n)\right)}{1 - \Phi(2-n^{1/2}\mu_n)} \sim \Unif(0,1),\\
&\hspace{50pt} n^{1/2}\bar{X}_{n, obs} > 2, ~X_{n,obs} \sim N\left(\mu_n, \frac{1}{n}\right)
\end{aligned}
\end{equation}

The pivotal quantity in \eqref{eq:pivot:univariate} allows us to construct $p$-values
or confidence intervals for Gaussian families. When the distributions $\F_n$'s are not
normal distributions, central limit theorem states that the sample mean $\bar{X}_n$ is
asymptotically normal when $\F_n$ has second moments.
Thus a natural question is whether the pivotal
quantity in \eqref{eq:pivot:univariate} is asymptotically $\Unif(0,1)$ when $X_{i,n}$
does not come from a normal distribution?

The following lemma provides a negative answer to this question in the case when
$\F_n$ is a translated Bernoulli distribution that has a negative mean. Essentially when
the selection event $n^{1/2} \bar{X}_n > 2$ becomes a rare event with vanishing
probability, the pivotal quantity in \eqref{eq:pivot:univariate} no longer converges
to $\Unif(0,1)$. We defer the proof of the lemma to the appendix.

\begin{lemma}
\label{lem:counter:eg}
If $X_{i,n}$ takes values in $\{-1.5, 0.5\}$, with $\Parg{X_{i,n} = -1.5} = \Parg{X_{i,n} = 0.5} =0.5$.
Thus $\mu_n = -0.5$. Then the pivot in \eqref{eq:pivot:univariate} does not converge to $\Unif(0,1)$
$$
P(\bar{X}_n) \not\rightarrow \Unif(0,1),
$$
for the $\bar{X}_n$'s surviving the file drawer effect \eqref{eq:file:drawer}.
\end{lemma}

Randomized selection circumvents this problem. In the following, we propose
a randomized version of the ``file drawer problem''.
\begin{example}[File drawer problem, randomized ]
\label{example2}
We assume the same setup of a triangular array of observations $X_{i,n}$
as in Example \ref{example1}. But instead of reporting $\bar{X}_n$ when
it survives the file drawer effect \eqref{eq:file:drawer},  
we independently draw $\omega \sim G$, and only report $\bar{X}_n$ if 
\begin{equation}
\label{eq:file:drawer:rand}
n^{1/2} \bar{X}_n + \omega > 2.
\end{equation}

Note that the selection event is different from that in \eqref{eq:file:drawer} in
that we randomize the sample mean before checking whether it passes the threshold.
In this case, if $\F_n = N(\mu_n, 1)$, the survival function of $\bar{X}_n$ is
\begin{equation}
\label{eq:pivot:uni:rand}
\begin{aligned}
P(t) &= \P\left( \bar{X}_n > t | n^{1/2} \bar{X}_n + \omega > 2\right), \quad (\bar{X}_n, \omega) \sim N\left(\mu_n, \frac{1}{n}\right) \times G\\ 
&=\P\left( Z  > n^{1/2}(t - \mu_n) | Z  + \omega > 2 - n^{1/2}\mu_n\right), \quad (Z, \omega) \sim N(0,1) \times G.
\end{aligned}
\end{equation}

To compute the exact form of $P(t)$, we have to compute the convolution of $N(0,1)$ and $G$ which
has explicit forms for many distributions $G$. Moreover, when $G$ is Logistic or Laplace distribution,
we have
$$
P(\bar{X}_{n, obs}) \rightarrow \Unif(0,1), 
$$
as long as $\F_n$ has centered exponential moments in
a fixed neighbourhood of $0$. The convergence is in fact uniform for $-\infty < \mu_n < \infty$.
For details, see Lemma \ref{lem:file:drawer} in Section \ref{sec:mean:CLT}.
\end{example}

The only difference between these two examples is the randomization in selection.
After selection, we need to consider the conditional distribution for inference,
which conditions on the selection event.
If we denote by $\F_n^*$ the distribution used for selective inference, we have
in Example \ref{example1}, 
\begin{equation}
\label{eq:likelihood1}
\frac{d\F_n^*}{d\F_n}(\bar{X}_n)  = \frac{1_{\{n^{1/2}\bar{X}_n > 2\}}}{\P_{\F_n}(n^{1/2}\bar{X}_n > 2)}.
\end{equation}
We also call the ratio between $\F^*_n$ and $\F_n$ the {\em selective likelihood ratio}.
In this case, the selective likelihood ratio is simply a restriction to the $\bar{X}_n$'s
that survives the file drawer effect. 
We observe that
$$
\sqrt{n}\bar{X}_n = \sqrt{n}\mu_n + Z, \quad Z \sim N(0,1), 
$$
which leads to three scenarios for selection.

\begin{itemize}
\item {\bf $\mu_n > \delta > 0$}, for some $\delta > 0$. 

In this case, the dominant term for selection is
$\sqrt{n}\mu_n$, and since we have a big positive effect, we would always report
the sample mean $\bar{X}_n$ when $n$ is big. This corresponds to the selection
event having probability tending to $1$ and the selective likelihood ratio goes
to $1$ as well. In this case, there is very little selection bias, and the original
law is a good approximation to the selective distribution for valid inference.

\item {\bf $\mu_n < -\delta < 0$}, for some $\delta > 0$. 

In this case, the dominant term is also $\sqrt{n}\mu_n$,
but in the negative direction. As $n \to \infty$, the selection probability vanishes
and the selective likelihood becomes degenerate.
We almost never report the sample mean in this scenario, but in the rare event where we do,
by no means can we use the original distribution for inference.

\item {\bf $-\delta < n^{1/2}\mu_n < \delta$}, for some $\delta > 0$.

This corresponds to local alternatives.
In this case, the selective likelihood neither converges to $1$ or becomes degenerate.
Rather, it becomes an indicator function of a half interval. Proper adjustment is needed
for valid inference in this case.
\end{itemize}

It is in the second scenario that pivotal quantity \eqref{eq:pivot:univariate} will not
converge to $\Unif(0,1)$. Different distributions will have different behaviors in the
tail. Since the conditioning event $n^{1/2}\bar{X}_n > 2$ becomes a large-deviations event,
we cannot expect it to behave like the normal distribution in the tail.

On the other hand, in Example \ref{example2},
if we denote by $\tilde{\F}_n^*$ the law for selective inference, we have
\begin{equation}
\label{eq:likelihood2}
\frac{d\tilde{\F}_n^*}{d\F_n}(\bar{X}_n) = \frac{\bar{G}(2 - n^{1/2}\bar{X}_n)}{\E_{\F_n}(\bar{G}(2 - n^{1/2}\bar{X}_n))}
= \frac{\bar{G}\left(2 - n^{1/2}(\bar{X}_n - \mu_n) - n^{1/2}\mu_n\right)}{\E_{\F_n}\left[\bar{G}\left(2 - n^{1/2}(\bar{X}_n - \mu_n) - n^{1/2}\mu_n\right)\right]}
\end{equation}
where $\bar{G}(t) = \int_t^{\infty} G(du)$ is the survival function of $G$. When
$\mu_n < -\delta < 0$ for some $\delta > 0$, and $G$ is the Laplace or Logistic distribution
so that $\bar{G}$ has an exponential tail, the dominant term $\exp(n^{1/2} \mu_n)$ in both
the numerator and the denominator will cancel out, making the selective likelihood
ratio properly behaved in this difficult scenario.

It turns out that this selective likelihood ratio is fundamental to formalizing
asymptotic properties of selective inference procedures. Its behavior
determines not only the asymptotic convergence of the pivotal quantities like in
\eqref{eq:pivot:uni:rand}, but also whether consistent estimation of the
population parameters is possible with large samples.

Again in the negative mean scenario where $\mu_n < -\delta < 0$, the sample mean $\bar{X}_n$
surviving the non-randomized ``file drawer effect'' cannot be a consistent estimator
for the underlying means $\mu_n$ because it will always be positive.
But if $\bar{X}_n$ is reported as in Example \ref{example2}, 
it will be consistent for $\mu_n$ even if $\mu_n$ is negative and
bounded away from $0$. For detailed discussion, see Section \ref{sec:nonparametric_setup}.

In general, the behavior of the selective likelihood ratio can be used to
study the asymptotic properties of selective inference procedures.
We study consistent estimation and weak convergence for selective inference procedures
in Section \ref{sec:nonparametric_setup} and Section \ref{sec:CLT} respectively. 

We are especially inspired by the field of differential privacy (c.f.
\cite{reusable_holdout} and references therein) to study the use of
randomization in selective inference. Privatized algorithms purposely randomize
reports from queries to a database in order to allow valid interactive data
analysis. To our understanding, our results are the first results related to
weak convergence in privatized algorithms, as most guarantees provided in the
differentially private literature are consistency guarantees. Some other
asymptotic results in selective inference have also been considered in
\cite{tibshirani2015uniform, tian2015asymptotics}, though these have a slightly
different flavor in that they marginalize over choices of models.

We conclude this section with some more examples.

\subsection{Linear regression}

Consider the linear regression framework with response $y \in \real^n$, and 
feature matrix $X \in \real^{n \times p}$, with $X$ fixed. We make a
homoscedasticity assumption that
$
\Covarg{y | X} = \sigma^2 I,
$
with $\sigma^2$ considered known.
Of interest is 
$$
\mu = \Ee(y|X),
$$
a functional of $\F=\F(X)$ the conditional law of $y$ given $X$. 
When $\F$ is a Gaussian distribution, exact selective tests have 
been proposed for different selection procedures \cite{lasso, spacings, tian2015selective}.
Removing the Gaussian distribution on $\F$, \cite{tian2015asymptotics}
showed that the same tests are asymptotically valid under some conditions.

Randomized selection in this setting is a natural extension of these works. \cite{optimal_inference}
proposed to use a subset of data for model selection, which yields 
a significant increase in power. In this work, we study general randomized selection procedures.
Consider the following example.

Due to the sparsity of the solution of LASSO \cite{lasso}
$$
\hat{\beta}_{\lambda}(y) = \argmin_{\beta \in \real^p} \frac{1}{2} \|y-X\beta\|^2_2 + \lambda \cdot \|\beta\|_1,
$$
a small subset of variables can be chosen for which we want to report $p$-values or confidence intervals.
This problem has been studied in \cite{exact_lasso}. However, instead of using the original response $y$
to select the variables, we can independently draw $\omega \sim \Q$ and choose the variables using $y^* = y + \omega$.
Specifically, we choose subset $E$ by solving
\begin{equation}
\label{eq:lasso:rand}
\hat{\beta}_{\lambda}(y, \omega) = \argmin_{\beta \in \real^p} \frac{1}{2} \|y^*-X\beta\|^2_2 + \lambda \cdot \|\beta\|_1,\quad y^* = y+\omega,
\end{equation}
and take $E = \supp{\hat{\beta}_{\lambda}(y, \omega)}$. In Section \ref{sec:linear:additive}, we discuss how
to carry out inference after this selection procedure, with much increased power. We also discuss the
reason behind this increase in Section \ref{sec:regression:additive}.

\subsection{Nonparametric selective inference}
\label{sec:nonparametric}

All the previous works on selective inference assume a parametric model like the Gaussian
family or the exponential family. In this work, we allow selective inference in a non-parametric
setting. Consider the following examples.

Suppose in a classification problem, we observe independent samples,
$$
(x_i, y_i) \iid \F, \quad (x_i, y_i) \in \real^p \times \{0,1\}.
$$
with fixed $p$. This problem is non-parametric if 
we do not assume any parametric structure for $\F$ and are simply interested in some population parameters
of the distribution $\F$. In Section \ref{sec:CLT},
we developed asymptotic theory to construct an asymptotically valid test for the population parameters of
interest. More details can be found in Section \ref{sec:randomized_logistic}.

Also consider a multi-group problem where a response $x$ is measured on $p$
treatment groups. A special case is the two-sample problem where there are two
groups.  It is of interest to form a confidence interval for the effect size in
the ``best'' treatment group.  This arises often in medical experiments where
multiple treatments are performed and we are interested to discover whether one
of the treatment has a positive effect. The fact we have chosen to report the
``best'' treatment effect exposes us to selection bias and multiple testing
issues \cite{benjamini_fdr}, and therefore calls for adjustment after
selection.  \citet{benjamini_stark} have considered the parametric setting
where $x_j \iid N(\mu_j, \sigma^2)$ for each group.  Suppose for robustness, it
is of interest to report the median effect size instead of the mean (assuming
responses are not symmetric). Then without any assumptions on the distribution
of the measurements, this also becomes a nonparametric problem. But we can
apply the theory in Section \ref{sec:CLT} to cope with this problem, for
details, see Section \ref{sec:median}.

\subsection{Outline of the paper}

There are three main advantages of applying randomization for selective inference,
\begin{itemize}
\item Consistent estimation under the selective distribution
\item Increase in power for selective tests
\item Weak convergence of selective inference procedures
\end{itemize}
 
In the following sections, Section \ref{sec:setup} gives the setup of selective inference and introduced
selective likelihood ratio, which is the key for studying consistent estimation and weak convergence of
selective inference procedures. Section \ref{sec:exact_inference} focuses on linear regression models with different randomization schemes,
demonstrating the increase in power.
Section \ref{sec:CLT} proposes an asymptotic test for the nonparametric
settings. Theorem \ref{thm:weak:conv} proves that the central limit theorem holds under the selective distribution with
mild conditions. Applications to the two examples in Section \ref{sec:nonparametric} are discussed.
This is a result for fixed dimension $p$. Finally, Section \ref{sec:multiple} discusses the possibility of
extending our work to the setting, when multiple selection procedures are performed on different randomizations
of the original data. One application is selective inference after cross validation for the square-root LASSO \cite{sqrt_lasso}.

\section{Selective Likelihood Ratio}
\label{sec:setup}

We first review some key concepts of selective inference. 
Our data $D$ lies in some measurable space $(\cD, \sF)$, with unknown sampling distribution $D \sim \F$. 
Selective inference seeks a reasonable probability model $M$ -- a subset of the probability measures on $(\cD,\sF)$,
and carry out inference in $M$. Central to our discussion is a {\em selection algorithm}, a set-valued map
\begin{equation}
\mQ: \cD \rightarrow {\cal Q}
\end{equation}
where ${\cal Q}$ is loosely defined as being made up of ``potentially interesting statistical questions''. 

For instance, in the linear regression setting, $\cD = \real^n$, our data $D = y$ and
we have a fixed feature matrix $X \in\real^{n \times p}$. The unknown sampling distribution
is $\F = \law(y|X)$, the conditional law of $y$ given $X$.

A reasonable candidate for the range of $\mQ$ might be all linear regression models
indexed by subsets of $\{1, \dots, p\}$ with known or unknown variance. For any selected
subset of variables $E$, we 
carry out selective inference within the model $M = \{N(X_E\beta_E, \sigma^2 I), \beta_E \in \real^{|E|}\}$. 

Since we use the data to choose the model $M$, it is only fair to consider the conditional distribution for inference,
$$
D | M \in \mQ(D), \quad D \sim \F.
$$
Therefore, we seek to control the selective Type-I error: 
\begin{equation}
\P_{M,H_0}( \text{reject $H_0$}\ |\ M  \in \mQ ) \leq \alpha
\label{eq:selective_type_1}
\end{equation}
where $M$ is the selected family of distributions in the range of $\mQ$ and $H_0 \subset M$ is the null hypothesis.  
Selective intervals for parametric models $M$ can then be constructed by inverting such selective hypothesis tests, though only the one-parameter
case has really been considered to date.

\subsection{Randomized selection}

Randomized selection is a natural extension of the framework above.
We enlarge our probability space to include some element of randomization. 
Specifically, let $\cH$ denote an auxiliary probability space and $\Q$ is a probability measure on $\cH$.
A randomized selection algorithm is then simply
$$
\mQ^*:\cD \times \cH \rightarrow \cQ.
$$

Note the randomization is completely under the control of the data analyst and
hence $\Q$ will be fully known.  This is an extension of the non-randomized
selective inference framework in the sense that we can take $\Q$ to be the
Dirac measure at $0$. Many choices of $\mQ^*$ are natural extensions of $\mQ$,
which we will see in many examples.

Randomized selective inference is simply based on the law $\F^*$, which we also call
the {\em selective distribution},
\begin{equation}
\label{eq:selective:law}
D | M \in \mQ^*(D,\omega), ~ (D, \omega) \sim \F \times \Q.
\end{equation}
Note that although randomization is incorporated into selection, inference is still carried
out using the original data $D$, after adjusting for the selection bias by considering the
conditional distribution $\F^*$.

Similar to the selective inference we defined above, we seek to control the selective
Type-I error, 
\begin{equation}
    \label{eq:random:typeIerror}
    \Pp_{\F^*}(\text{reject }H_0) = \Pp_{M,H_0}(\text{reject }H_0|M \in \mQ^*)\leq \alpha.
\end{equation}
Moreover, we also want to achieve good estimation, which makes
\begin{equation}
    \E_{\F^*}((\hat{\theta}(y)-\theta(\F))^2)
\end{equation}
small. 

In Sections \ref{sec:nonparametric_setup} to \ref{sec:CLT}, we will discuss concrete examples
of $\cD$, $D$, $\F$ and $\mQ^*$. But before that we first introduce the selective likelihood
ratio, which is a crucial quantity in studying the selective distribution $\F^*$.

\subsection{Selective likelihood ratio}

Selective likelihood ratio provides a way of connecting the original distribution $\F$ and its
selective counterpart $\F^*$. It is easy to see from \eqref{eq:selective:law} that the selective
distribution is simply a restriction of the $(D, \omega)$'s such that model $M$ will be selected.
Thus $\F^*$ is absolutely continuous with respect to $\F$, and the {\em selective likelihood ratio} is 
\begin{equation}
\label{eq:selective_dbn}
\begin{aligned}
\frac{d\F^*}{d\F}(D) &= \frac{\modelweight(M;D)}{\E_{\F}(\modelweight(M;D))} = \ell_{\F}(D) \qquad  \forall \, \F \in M, \\
\modelweight(M;D) &= \Q\left(\left\{\omega: M \in \mQ^*(D,\omega) \right\} \right). 
\end{aligned}
\end{equation}

The numerator in $\ell_{\F}(D)$ is the restriction of $(D, \omega)$, integrated over the randomizations $\omega$,
and the denominator is simply a normalizing constant. One implication of the selective likelihood ratio 
is that for distributions $\F$ in parametric families, their selective counterparts may have the
same parametric structure. 

\subsubsection{Exponential families}
\label{sec:exponential_family}

One commonly used parametric family is the exponential family.
Assume that $\F=\F_{\theta}$ is an exponential family with natural parameter space $\paramset$ and
$\cD = \real^n$ and the data $D = y$. Its density with respect to the reference measure $d\F_0$ is,
\begin{equation}
\label{eq:exfam}
    \frac{d\F_{\theta}}{d\F_0}(y) =  \exp\{\theta^T T(y) - \psi(\theta)\}, \quad \theta \in \paramset.
\end{equation}

Through the relationship in \eqref{eq:selective_dbn} we conclude, for any randomization
scheme, the law $\F^*_{M,\theta}$ is another exponential family. Formally,

\begin{lemma}
\label{lem:exponential:family}
If $\F_{\theta}$ belongs to the exponential family in \eqref{eq:exfam}, then for {\bf any}
randomized selection procedure $\mQ^*$, the selective distribution is also an exponential family,
$$
\frac{d\F^*_{M,\theta}}{d\F_0}(y) \propto \modelweight(M;y) \exp\{\theta^T T(y) - \psi(\theta)\}, \qquad \theta \in \paramset.
$$
with the same sufficient statistic $T(y)$ and natural parameters $\theta$.

Furthermore, to test $H_{0j}: \theta_j = 0$, we consider the following law,
\begin{equation}
\label{eq:conditional:expfam}
T_j(y) \mid T_{-j}(y), \qquad y \sim \F^*_{M,\theta}.
\end{equation}
\end{lemma}

The first claim of the lemma is quite straight-forward using the relationship in \eqref{eq:selective_dbn}.
The second claim is a Lehmann--Scheffe (c.f. Chapter 4.4 in \cite{TSH}) construction which was proposed in \cite{optimal_inference},
to construct tests for one of the natural parameters treating
the others as nuisance parameters. For detailed construction of such tests in the linear regression setting,
see Section \ref{sec:exact_inference}.

\section{Consistent Estimation After Model Selection}
\label{sec:nonparametric_setup}

In this section, we leave the parametric setup and consider general models $M$.
In particular, we study the consistency of estimators under the selective distribution
for arbitrary models. We first introduce the framework of asymptotic analysis under
the selective model. Then we state conditions for consistent estimation in Lemma   
\ref{lem:transfer} and conclude with examples.

For any model $M$, which is a collection of distributions,
we define its corresponding {\em selective model}, which is the collection of
corresponding selective distributions, 
\begin{equation}
\label{eq:selective:model}
M^* = \left\{\F^* : \frac{d\F^*}{d\F}(D) = \ell_{\F}(D), \F \in M \right\},
\end{equation}
where $\ell_{\F}(D)$ is the selective likelihood ratio for the selection
event $\{M \in \mQ^*\}$.
Selective inference is carried out under the selective model $M^*$.

In order to make meaningful asymptotic statements, we consider a sequence
of randomized selection procedures $(\mQ^*_n)_{n \geq 1}$ and models
$(M_n)_{n \geq 1}$ with each $M_n$ in the range of $\mQ^*_n$. 

Often, we are interested in some population parameter $\theta_n$, which
can be thought be as a functional of the distribution $\F_n \in M_n$,
$$
\theta_n:M_n \rightarrow \real.
$$
It is worth pointing out that $M_n$ is selected by $\mQ^*_n$, which already
incorporates the statistical questions we are interested in. In this sense,
$M_n$ is chosen a posteriori. The selected model $M_n^*$ does not change our
target of inference, it merely changes the distribution under which such inference
should be carried out. In other words, if $\theta_n$ is the mean parameter, we
are interested in the underlying mean of $\F_n$, not $\F_n^*$. 

We might have a good estimator $\hat{\theta}_n: \cD \to \real$ for $\theta_n(\F_n)$
under $\F_n$, namely
$$
\Esubarg{\F_n}{(\hat{\theta}_n - \theta_n(\F_n))^2} \to 0.
$$
$\hat{\theta}_n$ is a consistent estimator if our model $M_n$ is given a priori.
But as we use data select $M_n$, what really cares about is its performance under the selective
distribution $\F_n^*$. Will this estimator still be consistent
under the selective distribution $\F_n^*$?

Formally, we say an estimator $\hat{\theta}_n$ is uniformly consistent in $L^p$ for 
$\theta_n(\F_n)$ under the sequence $(M_n)_{n \geq 1}$ if
$$
\limsup_n \sup_{\F_n \in M_n} \|\hat{\theta}_n - \theta_n(\F_n)\|_{L^p(\F_n)} \to 0.
$$
Similarly, we say that $\hat{\theta}_n$ is uniformly consistent in probability for the
functional $\theta_n(\F_n)$ under
the sequence $(M_n)_{n \geq 1}$ if
for every $\epsilon > 0$ there exists $\delta(\epsilon) > 0$ such that for all $\delta \geq \delta(\epsilon)$
$$
\limsup_n \sup_{\F_n \in M_n} \F_n(|\hat{\theta}_n-\theta_n(\F_n)| > \delta) \leq \epsilon.
$$

The following lemma states the conditions for consistency of $\hat{\theta}_n$ under the sequence
of corresponding selective models $(M_n^*)_{n \geq 1}$,

\begin{lemma}
    \label{lem:transfer}
Consider a sequence $(\mQ_n^*,M_n)_{n \geq 1}$ 
of randomized selection procedures and models.
Suppose the selective likelihood ratios satisfies, for some $p>1$,
\begin{equation}
\label{eq:likelihood:Lp}
\limsup_n \sup_{\F_n \in M_n} \|\ell_{\F_n}\|_{L^p(\F_n)} < C.
\end{equation}
Then for any sequence of estimators $\hat{\theta}_n$ uniformly consistent for $\theta_n(\F_n)$ in $L^{\alpha}$,
it is also uniformly consistent for $\theta_n(\F_n)$ in $L^{\gamma}$ under
$(M_n^*)_{n \geq 1}$, $\gamma \leq \alpha / q$, $\frac{1}{p}+\frac{1}{q}=1$.

Further, if $\hat{\theta}_n$ is uniformly consistent for $\theta_n$ in probability, 
then $\hat{\theta}_n$ is uniformly consistent for $\theta_n$ in probability
under the sequence $(M_n^*)_{n \geq 1}$.
\end{lemma}

\begin{proof}
Let $\Delta_n=\hat{\theta}_n-\theta_n(\F_n)$. 
To prove the first assertion note that for any $\F^*_n \in M^*_n$
$$
\begin{aligned}
\|\Delta_n\|_{L^{\gamma}(\F^*_n)} &= \int_{\cD_n} |\Delta_n|^{\gamma} \ell_{\F_n}(y) \F_n(dy) \\
& \leq \||\Delta_n|^{\gamma}\|_{L^q(\F_n)} \| \ell_{\F_n}(y) \|_{L^p(\F_n)} \\
& = \|\Delta_n\|_{L^{\gamma q}(\F_n)}^{\gamma} \| \ell_{\F_n}(y) \|_{L^p(\F_n)} \\
& \leq \|\Delta_n\|_{L^{\alpha}(\F_n)}^{\gamma} \| \ell_{\F_n}(y) \|_{L^p(\F_n)} \\
& \leq C\|\Delta_n\|_{L^{\alpha}(\F_n)}^{\gamma}
\end{aligned}
$$
    For any $\delta > 0$,
    $$
    \begin{aligned}
    \F_n^*(|\Delta_n|>\delta) &= \int_{\cD_n} \mathbf{1}\{|\Delta_n| > \delta\} \ell_{\F_n}(y) \F_n(dy) \\
    &\leq \left[\F_n(|\Delta_n|>\delta)\right]^{1/q}\|\ell_{\F_n}\|_{L^p(\F)} \\
    &\leq C\left[\F_n(|\Delta_n|>\delta)\right]^{1/q}.
    \end{aligned}
    $$
\end{proof}

We illustrate the application of Lemma \ref{lem:transfer} through our ``file drawer effect'' 
examples in Section \ref{sec:file:drawer}. 

\subsection{Revisit the ``file drawer problem''}

First we note that in Example \ref{example1} and \ref{example2}, we observe data
$D_n = (X_{1, n}, \dots, X_{n,n})$, with $X_{i,n} \sim \F_n$. The randomized selection in
Example \ref{example2} can be realized as
$$
\mQ^*(D_n, \omega) = \begin{cases} 
\text{report p-values for } \bar{X}_n, &\text{  if} \sqrt{n}\bar{X}_n + \omega > 2, \\
\text{do nothing}, &\text{  if} \sqrt{n}\bar{X}_n + \omega \leq 2,
\end{cases}
$$
where we independently draw $\omega \sim G$.

By law of large numbers, we easily see that if we always report $\bar{X}_n$, it will
be an unbiased estimator for $\mu_n$. However, 
since we only observe the sample means surviving the file drawer effect. Will $\bar{X}_n$
still be consistent for $\mu_n$?

In the most difficult scenario discussed in Section \ref{sec:file:drawer}, where
$\mu_n < -\delta <0$ for some $\delta > 0$, $\bar{X}_n$ cannot be a consistent estimator
for $\mu_n$ in Example \ref{example1}. This is easy to see as Example \ref{example1} will only
report positive sample means. A remarkable feature of randomized selection is that
consistent estimation of the population parameters is possible even when the selection
event has vanishing probabilities. In fact, the following lemma states that when
$G$ is a Logistic distribution, $\bar{X}_n$ is consistent for $\mu_n$ after the
randomized file drawer effect in Example \ref{example2}. 

\begin{lemma}
\label{lem:simple:consistent}
Suppose as in Example \ref{example2}, we observe a triangular array with $X_{i,n} \sim \F_n$.
$\F_n$ has mean $\mu_n = \mu < 0$. If we draw $\omega \sim \Logistic(\kappa)$, where $\kappa$
is the scale of the Logistic distribution. Then the sample means $\bar{X}_n$ surviving
the ``randomized'' file drawer effect are consistent for $\mu$,
$$
\bar{X}_n \overset{p}{\to} \mu, \quad \text{conditional on }\sqrt{n}\bar{X}_n + \omega > 2.
$$
if $\F_n$ has moment generating function in a neighbourhood of $0$. Namely, $\exists a>0$, such that
$$
\Esubarg{\F_n}{\exp\left(a \abv{X_{i, n} - \mu_n}\right)} \leq C.
$$
\end{lemma}

Before we prove the lemma, we want to point out that although the selection procedure in
Example \ref{example2} is different from that in Example \ref{example1} because of randomization,
$\sqrt{n}\mu_n$ is still the dominant term in selection. Note that
$$
\sqrt{n}\bar{X}_n + \omega = \sqrt{n}\mu_n + \sqrt{n}(\bar{X}_n -\mu_n) + \omega.
$$
Since both $\sqrt{n}(\bar{X}_n -\mu_n)$ and $\omega$ are $O_p(1)$ random variables, the
dominant term $\sqrt{n}\mu_n \rightarrow -\infty$, would ensure that the selection event
has vanishing probabilities in Example \ref{example2} as well. Thus it is particularly
impressive that Example \ref{example2} gives consistent estimation where Example \ref{example1}
cannot. The proof of Lemma \ref{lem:simple:consistent} is deferred to the appendix.

We also verified this theory of consistent estimation through simulations. Figure \ref{fig:consistency}
shows the empirical distributions of the sample mean $\bar{X}_n$ after the file drawer effect in
Example \ref{example1} or the ``randomized'' file drawer effect in Example \ref{example2}. 
They are marked with ``blue'' colors or ``red'' colors respectively. We set the
true underlying mean to be $\mu_n = \mu = -1$ and mark it with the dotted vertical line in Figure
\ref{fig:consistency}. The upper panel Figure \ref{fig:100} is simulated with $n=100$ and the
lower panel Figure \ref{fig:250} is simulated with $n=250$. We notice that in both simulations,
the sample mean in Example \ref{example1} concentrates around the thresholding boundary, which
is positive. Thus, these sample means can not be possibly for the underlying mean $\mu = -1$. 
However, the existence of randomization allows us to report negative sample means. As a result,
the sample mean in Example \ref{example2} will be consistent for $\mu = -1$. We see that as
we increase sample size $n$, the sample means concentrates closer to $\mu=-1$. 

\begin{figure}
    \centering
    \begin{subfigure}[b]{\textwidth}
        \includegraphics[width=\textwidth]{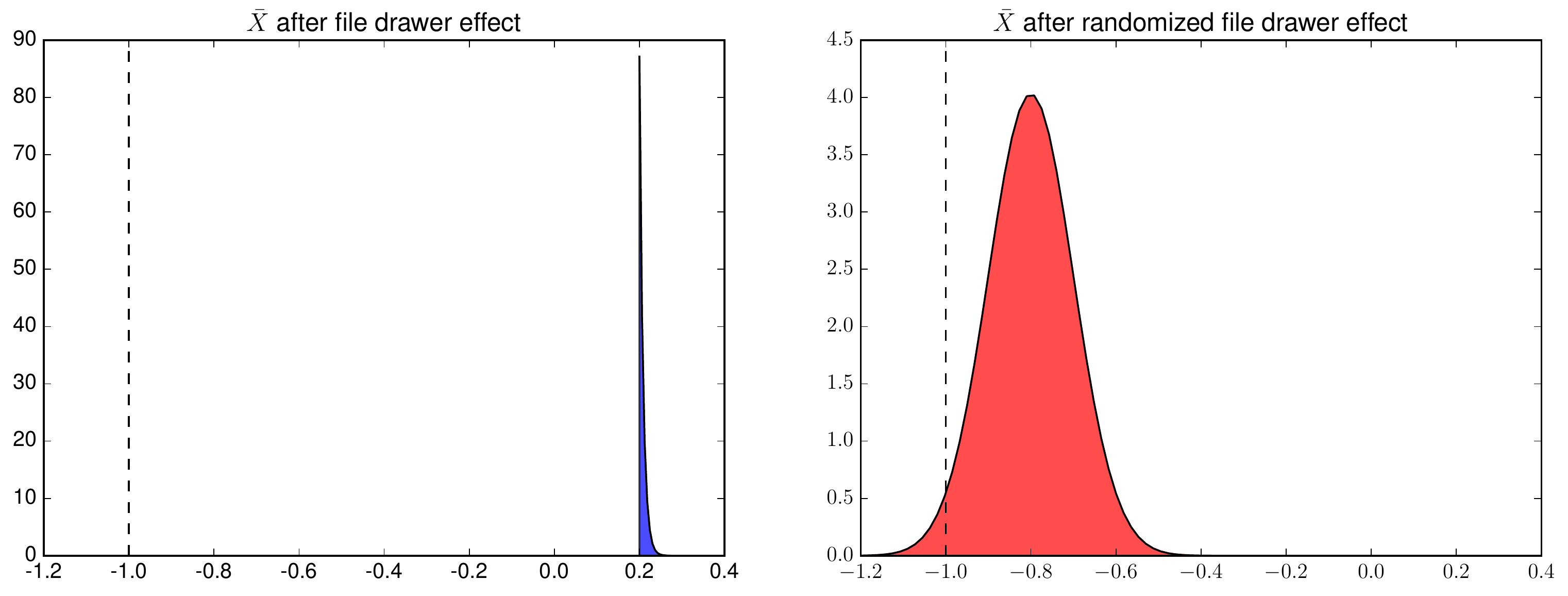}
        \caption{$n = 100$} \label{fig:100}
    \end{subfigure}
    ~ 
    \begin{subfigure}[b]{\textwidth}
        \includegraphics[width=\textwidth]{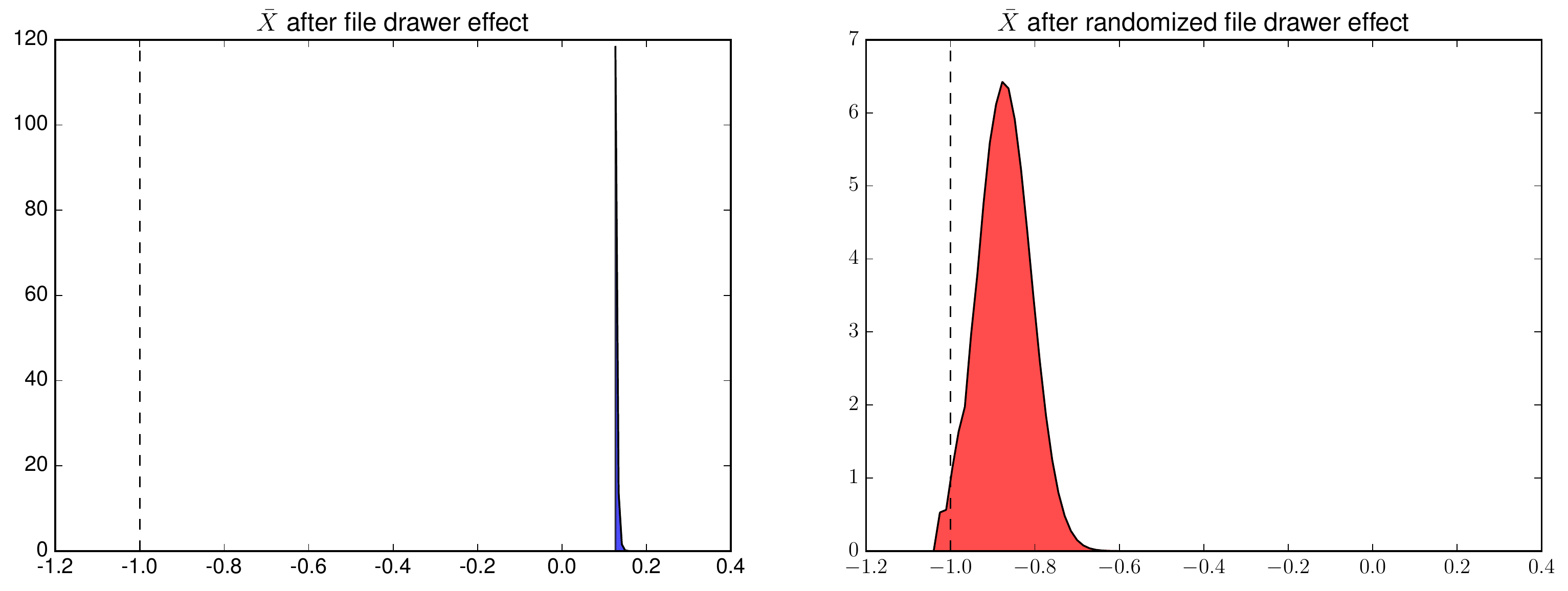}
        \caption{$n = 250$} \label{fig:250}
    \end{subfigure}
    \caption{Empirical distributions of sample means $\bar{X}_n$ in Example \ref{example1} and Example \ref{example2},
    with original or randomized file drawer effect. For the randomization, we draw $\omega \sim \Logistic(\kappa)$,
    with $\kappa = 0.5$.}\label{fig:consistency}
\end{figure}

\section{Inference in linear regression models}
\label{sec:exact_inference}

In the linear regression setting, we assume a fixed feature matrix $X \in \real^{n \times p}$, and observe
the response vector $D = y \in \real^n$. We assume the noises are normally distributed. 
There are two ways to parametrize a linear model, and both belong to some exponential family. Now we introduce
the {\em selected model},
\begin{equation}
\label{eq:selected}
M_{sel}(E) = \left\{N(X_E\beta_E, \sigma^2 I): \beta_E \in \real^{|E|}\right\}, \ \  E \subset \{1, \dots, p\}
\end{equation}
with $\sigma^2$ known or unknown or the {\em saturated model},
\begin{equation}
\label{eq:saturated}
M_{sat} = \left\{N(\mu, \sigma^2 I): \mu \in \real^n \right\}
\end{equation}
with known variance. Now we consider some randomized selection procedures and inference after selection.
 
\subsection{Data splitting and data carving}

In the introduction, we introduced {\em data splitting} \cite{data_splitting} as a special case
of randomized selective inference. 
In \cite{optimal_inference}, the term {\em data carving} was introduced to demonstrate that 
data splitting is inadmissible. 
In data splitting (and data carving) inference
makes most sense in the selected model $M_{sel}(E)$,
hence we should think of $\mQ$ as returning a subset $E$ of variables selected.

Let us formalize this notion
in our notation. Let $\Q$ be some measure on assignments of $n$ data points into groups and $\mQ$ a selection
algorithm defined on datasets of any size. The distribution $\Q$ determines
a randomized selective inference procedure with selection algorithm $\mQ^*$, an 
algorithm applied to subsets of the original data set. In this case, it is easy to see that
$$
\modelweight(E;y) \overset{D}{=} \modelweight(M_{sel}(E);y) \propto \sum_{\omega} q_{\omega} \cdot 1_{\{M_{sel}(E) \in \mQ(y_1(y,\omega))\}}
$$
where $q_{\omega}$ is the mass assigned to assignment $\omega$ by $\Q$. Multiple assignments or splits
considered in \cite{meinshausen_pvals,stability_selection} can be formalized in a similar fashion. 
We can construct UMPU tests for $\beta_E$ in the selected model 
$M_{sel}(E)$ by using Lemma \ref{lem:exponential:family}, (also see \cite{optimal_inference}).
We note that in \cite{optimal_inference} the authors conditioned unnecessarily 
on the split $\omega$, and we would expect that aggregating over splits would yield a more powerful procedure.

However, there are two disadvantages with this randomization scheme. First, it is computationally difficult 
to aggregate over all random splits. Second, it seems difficult to consider the saturated model $M_{sat}$ for
inference, which is more robust to model misspecifications. To overcome those difficulties, we introduce other
randomization schemes below.

\subsection{Additive noise and more powerful tests}
\label{sec:regression:additive}

Our second randomization scheme in linear regression involves additive noise.
Specifically, we draw $\omega \sim \Q$ and use the randomized response $y^*(y,\omega) = y+\omega$ for selection 
In this case, we can consider both the selected model $M_{sel, E}$ and the saturated model $\M_{sat}$. 
Per Lemma \ref{lem:exponential:family}, we can perform valid inference
for $\beta_E$ in $M_{sel,E}$ or linear functionals of $\mu$ in $M_{sat}$.

One major advantage of using a randomized response $y^*$ for selective inference is
that these procedures yield much more powerful tests, at a small cost of on the quality of the
selected models. In other words, small amount of randomization is cause a small loss in
the model selection stage, but we gain much more power in the inference stage. 

The reason for increased power can be explained by a notion called {\em leftover Fisher Information}
first introduced in \cite{optimal_inference}. Since selective inference is essentially inference
under the selective distribution $\F_n^*$, the Fisher Information under $\F_n^*$ would determine
how efficient the selective tests are. In the saturated model with Gaussian noise $M_{sat}$, 
$\frac{y-\mu}{\sigma^2}$ is the score statistic and its variance under $\F_n^*$ is exactly
the leftover Fisher Information (a similar relationship holds in the selected model $M_{sel,E}$). 
Lemma \ref{lem:Fisher} gives a lower bound on this leftover Fisher Information when
the randomization noise $\Q = N(0, \gamma^2 I)$.

\begin{lemma}
    \label{lem:Fisher}
For either $M_{sat}$ or $M_{sel}(E)$, if we use Gaussian randomization noise $\Q = N(0, \gamma^2)$,
and the selection is based on $\mQ(y^*) = \mQ(y+\omega)$,
then the leftover Fisher information is bounded below by
$$
(1 - \tau){\cal I}(\theta), \quad  \tau=\sigma^2 / (\sigma^2 + \gamma^2),
$$
and ${\cal I}(\theta)$ is the non-selective Fisher information
for $\theta$ in $M_{sat}$ or $M_{sel}(E)$. 
The parameters $\theta$ depend on which of the two models we are considering.
\end{lemma}

\begin{proof}
In the saturated model $M_{sat}$, the score statistic is
$
V = \frac{y - \mu}{\sigma^2}.
$
Since $\mQ(y^*)$ is measurable with respect to $y^*$,
$$
\Vararg{V \mid \mQ(y^*)} \geq \Vararg{V \mid y^*} = \frac{1}{\sigma^4}\Vararg{y \mid y^*}.
$$
Since $y$ and $y^* = y + \omega$ are both normal distributions with covariance matrices,
$$
\Covarg{y, y^*} = \sigma^2 I, \Vararg{y^*} = (\gamma^2 + \sigma^2) I, 
$$
we have the leftover Fisher Information
$$
\begin{aligned}
&\Vararg{V \mid \mQ(y^*)} 
 \geq \frac{1}{\sigma^4}\Vararg{y \mid y^*} \\
= &\frac{1}{\sigma^4} (\sigma^2 I - \frac{\sigma^4}{\gamma^2 + \sigma^2} I)
=  \frac{1}{\sigma^2} (1 - \tau) I = (1 - \tau) {\cal I}(\mu). 
\end{aligned}
$$

In the selected model $M_{sel, E}$, the score statistic is
$
V = \frac{X_E^T(y - X_E \beta_E)}{\sigma^2}.
$
Similarly,
$$
\begin{aligned}
&\Vararg{V \mid \mQ(y^*)} 
\geq \frac{1}{\sigma^4}\Vararg{X_E^T y \mid y^*} \\
=& \frac{1}{\sigma^4} \left[\sigma^2(X_E^T X_E) - \frac{\sigma^4}{\gamma^2 + \sigma^2} (X_E^T X_E)\right]
= \frac{1}{\sigma^2} (1 - \tau) X_E^T X_E = (1 - \tau) {\cal I}(\beta_E). 
\end{aligned}
$$
\end{proof}

When there is no randomization $\gamma=0$, we potentially have no leftover Fisher information. This corresponds
to a very rare selection event. However after randomization, even with very extreme selection, there is always
leftover Fisher information, which makes the selective tests more powerful. Consider the following examples.

\subsubsection{Revisit the ``file drawer problem''}

In Example \ref{example1} and Example \ref{example2}, if we assume $\F_n = N(\mu, 1)$, they
are a special case of the linear regression model, with the feature matrix $X = \mathbf{1}$,
the all ones vector.

In this case, $n\bar{X}_n$ is the score statistic, and
its variance under the selective distribution
is the Fisher information. Lemma \ref{lem:Fisher}
states that the leftover Fisher information is lower bounded by $n(1-\tau)$ if we draw
randomize using Gaussian variables, $\Q = N(0, \gamma^2)$, $\tau = 1/(1+\gamma^2)$.

Moreover, the increase in leftover Fisher information with randomization is not
specific to Gaussian randomizations. For example, in Figure \ref{fig:consistency}
when we use Logistic randomization,
we also observe that under the selective distribution with randomization, $\bar{X}_n$ has a
much bigger variance than without randomization. As discussed above, this variance multiplied by $n^2$
is exactly the leftover Fisher information, which explains why selective procedures
after randomization will have better performances than without. 

We investigate the relationship between the leftover Fisher information and the
length of confidence intervals constructed by inverting the pivot in
\eqref{eq:pivot:uni:rand}. Specifically, in Example \ref{example2}, after
observing a reported sample mean, we want to report confidence intervals for
the underlying mean $\mu$.

Figure \ref{fig:CIs} demonstrates the selective intervals (solid lines) after 
\eqref{eq:file:drawer:rand} with $\omega$ being either Gaussian or Logistic noises. 
The sample size $n=100$. Unlike the nominal
confidence intervals (dashed lines), the selective intervals are valid with $90\%$
coverage for the underlying mean.  
Since Lemma \ref{lem:transfer} gives a lower bound of $(1 - \tau) {\cal I}(\mu)$,
we would intuitively expect the selective confidence intervals to be $1/(1-\tau)$
the length of the nominal intervals. This is verified in Figure \ref{fig:gaussian},
when we observe really negative sample means. (The sample means can be negative because we added randomization.)
On the other hand, for Logistic randomization in Figure \ref{fig:logistic},
the intervals are slightly wider than the nominal intervals around the $2/\sqrt{n}$, but narrow to roughly the nominal size on both sides of the truncation point.
This indicates that added logistic noise might preserve more information than Gaussian additive noise.
Both additive noises improve significantly over a non-randomization scheme (c.f. Figure
3 in \cite{optimal_inference}).

\begin{figure} 
    \centering
    \begin{subfigure}[t]{0.45\textwidth}
                \includegraphics[width=\textwidth]{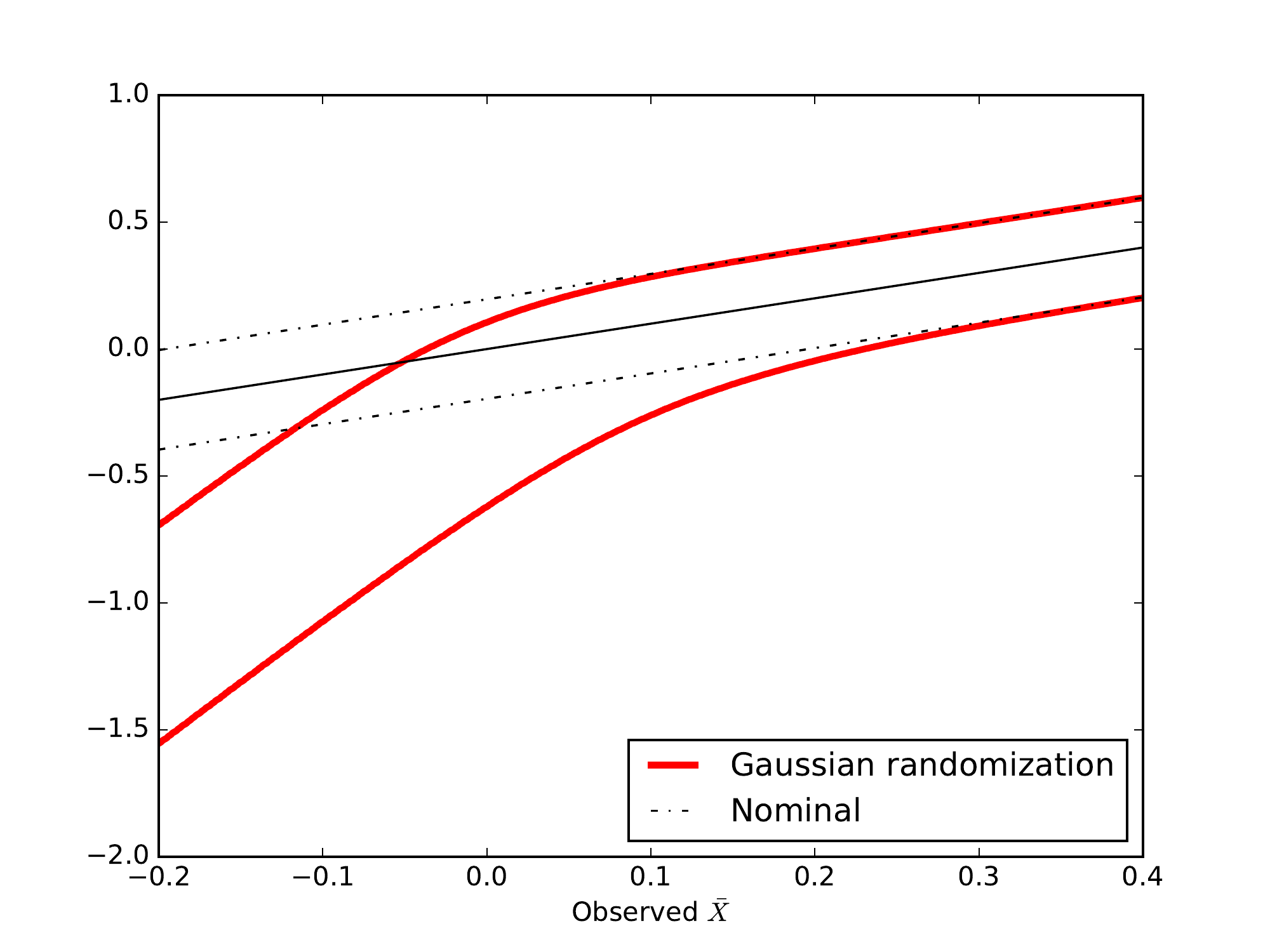}
                \caption{Gaussian added noise}
                \label{fig:gaussian}
        \end{subfigure}%
        ~ 
        \begin{subfigure}[t]{0.45\textwidth}
                \includegraphics[width=\textwidth]{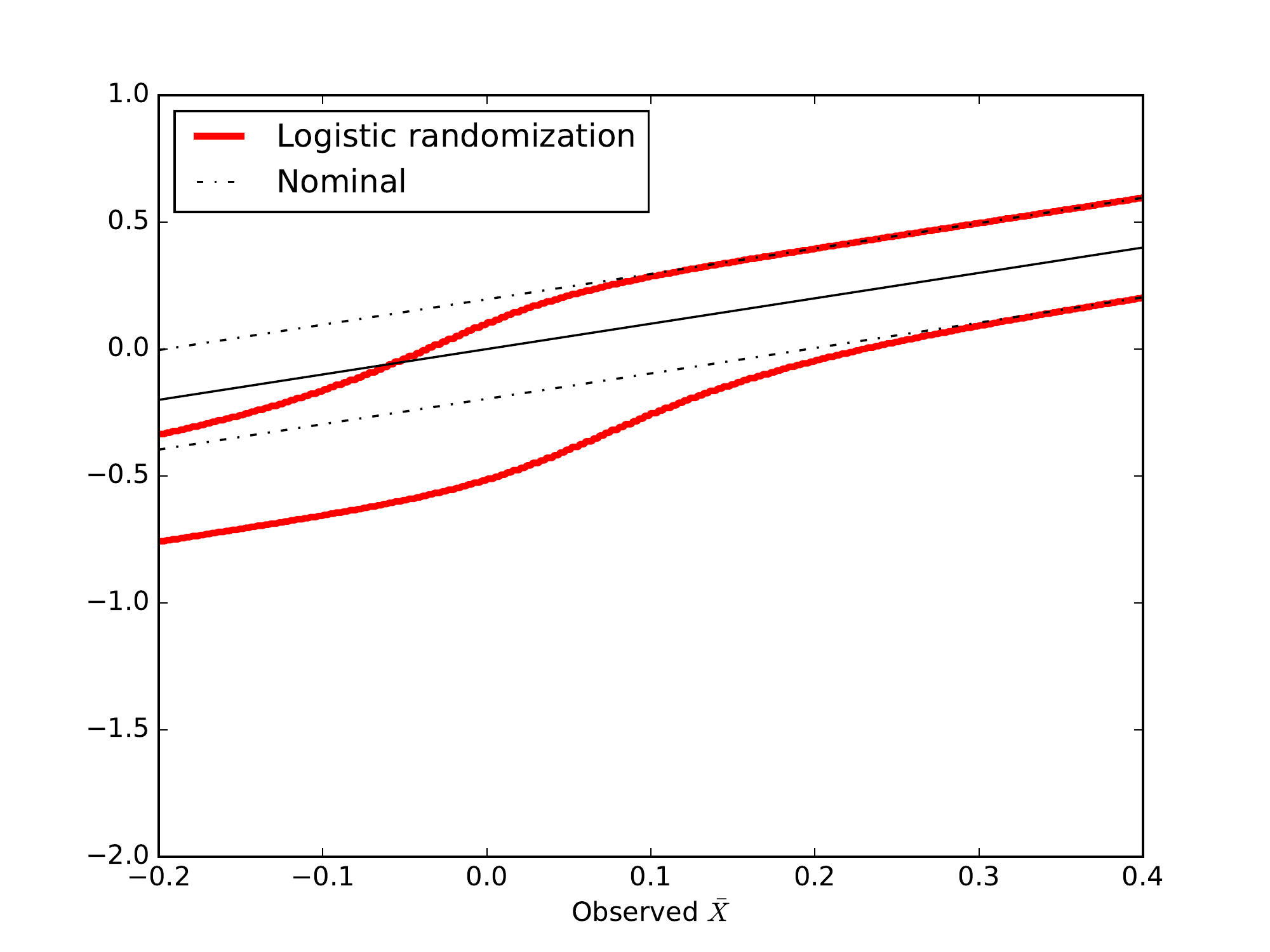}
                \caption{logistic added noise}
                \label{fig:logistic}
        \end{subfigure}
        \caption{Selective confidence intervals for different added noise}\label{fig:CIs}
\end{figure}

Of course, the increase in power and shortening of selective confidence intervals does not
come without a price. Because we select with a randomized response, we are likely to
select a worse model. But the trade-off between model quality and power is highly in
favor of randomization. See the following example.

\subsubsection{Linear regression with added noise}
\label{sec:linear:additive}

Back to the general setup of linear regression models, we select a model by
solving LASSO with the randomized response $y^* = y + \omega$ and return the
active set $E$ of the solution (as in \eqref{eq:lasso:rand}). Then per Lemma \ref{lem:exponential:family},
we can construct valid selective tests in both $M_{sat}$ and $M_{sel}(E)$. For
instance, in $M_{sel}(E)$, we can construct
tests for the hypothesis $H_{0j}: \beta_j = 0, ~ j \in E$ based on the law,
\begin{equation}
\label{eq:poly:inference:law}
\eta^T y  \bigl \vert A_E(y + \omega) \leq b_E, P_{E \backslash j}y, \quad (y, \omega) \sim N(X_E\beta_E,\sigma^2I) \times \Q,~\beta_j = 0,
\end{equation}
where $\eta = (X_E^{\dagger})^T e_j$, $e_j$ is the $j$-th column of the identity matrix, 
$P_{E \backslash j}$ is the projection matrix onto the column space of $E$ but orthogonal to $\eta$,
and $A_E,~b_E$ are the appropriate matrix and vector corresponding to LASSO selection. 
This is a UMPU test due to the Lehmann--Scheffe construction \citep{optimal_inference}
and controls the selective Type-I error \eqref{eq:random:typeIerror}. Although, we cannot
compute the explicit forms of \eqref{eq:poly:inference:law}, the selection events in
\eqref{eq:poly:inference:law} are polyhedrons and thus a hit-and-run or Hamiltonian Monte Carlo algorithm \cite{paninski} 
can be used for sampling.

Figure \ref{fig:ROC} compares inference in the additive Gaussian noise scheme
to the data carving procedure proposed in \cite{optimal_inference} as well as
data splitting.  In $M_{sel}(E)$, the probability of screening (i.e. selecting
$E$ including all the nonzero $\beta$'s) is a surrogate for the quality of the
model. As additive noise uses a different randomization scheme than data splitting
and data carving, we vary the amount of randomization used in each scheme and match
on the probability of screening. Thus Figure \ref{fig:ROC} is like an ROC curve for
the trade-off between model quality and power of tests. The $x$-axis goes in the
direction of increased randomization, with the left most point corresponding to
no randomization at all. We see even with a small randomization
that barely affects model selection, we can substantially lower the Type-II error
from $0.2$ to less than $0.05$. The trade-off is highly in favor of (small) randomization.
We see in Figure \ref{fig:ROC} that additive noise lowers the Type-II error by almost half than data carving
for the same screening probability and they both clearly dominate data splitting. 
For the concrete setup of the simulation, see Chapter 7 of \cite{optimal_inference}. 

\begin{figure}
\begin{center}
\includegraphics[width=.55\textwidth]{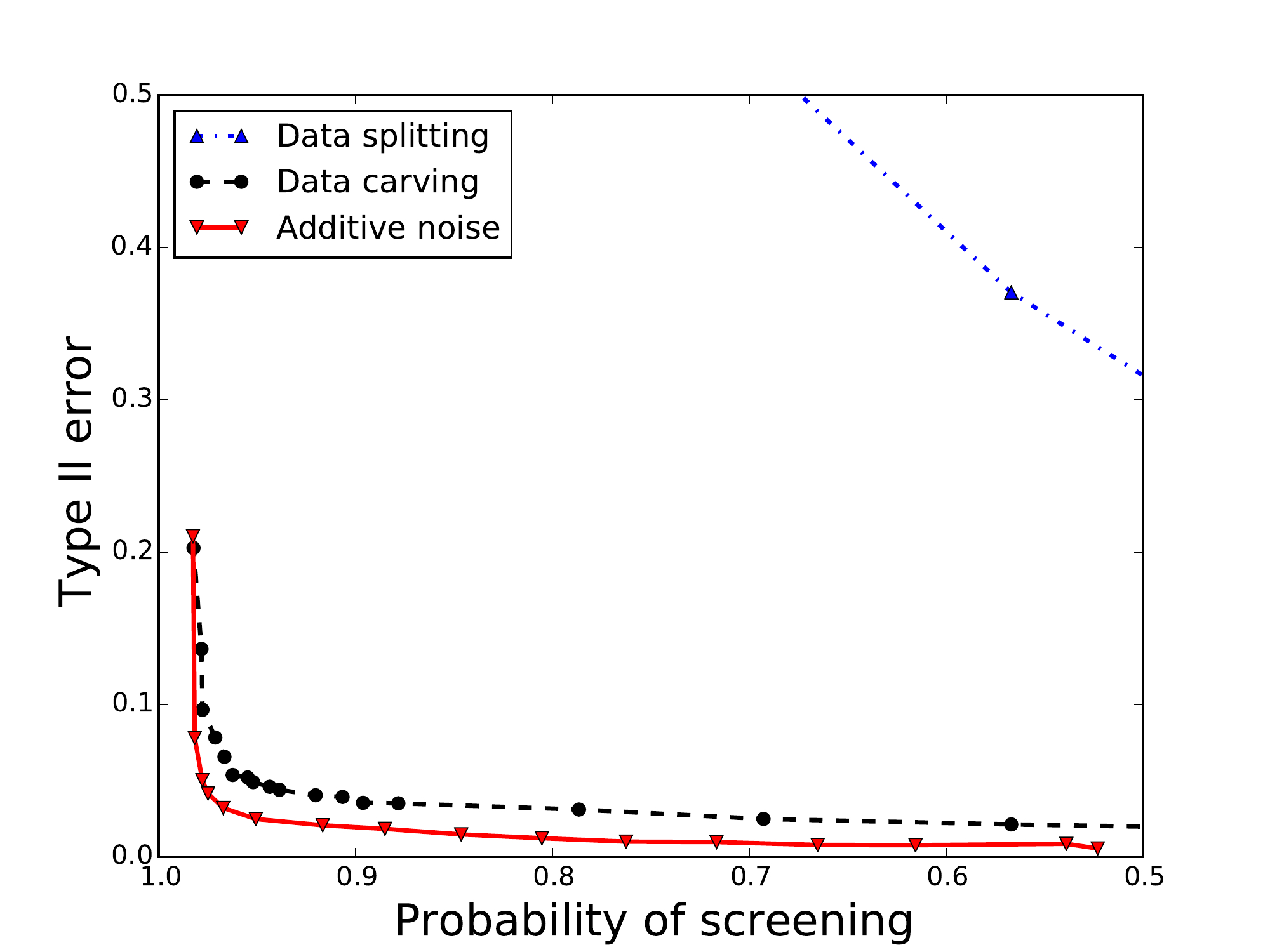}
\end{center}
\caption{Comparison of inference in additive noise randomization vs. data carving.}
\label{fig:ROC}
\end{figure}

\section{Weak convergence and selective inference for statistical functionals}
\label{sec:CLT}

\newcommand{\stat}{T}
\newcommand{\nuisance}{N}
\newcommand{\normalize}{{\cal N}}

In the nonparametric setting, we assume a triangular array of data, $D_n = (d_{1,n},\dots, d_{n,n})$,
and $d_{i,n} \iid \F_n$. When $\F_n = \F$, it is the special case of independent sampling.
We are interested in some functional of the distribution $\mu_n = \mu(\F_n)$. Associated
with $\mu_n$ is our statistic $\stat$ which is a {\em linearizable statistic} \citep{linearizable}. 
\begin{definition}[Linearizable statistic]
    \label{def:decomposable}
    Suppose $d_{i,n} \iid \F_n$, we call $T$ a linearizable statistic for $\mu_n=\mu(\F_n)$ if for any sample size $n$,
    \begin{equation}
    \label{eq:decomposable}
    \begin{aligned}
    &T(D_n) = \frac{1}{n} \sum_{i=1}^n \xi_{i,n} + R, \quad \xi_{i,n} = \xi(d_{i,n}), \\
    &\Earg{\xi_{i,n}} = \mu_n \in \real^p,  \quad \Covarg{\xi_{i,n}} = \Sigma_n \in \real^{p\times p}.
    \end{aligned}
    \end{equation}
    where $\xi$ a function of the data and $R$ is bounded with probability $1$, $R = o_p(n^{-\frac{1}{2}})$ under $\F$.
    We use the slight abuse of notations to denote $\xi_{i,n}$ as $\mrm{iid}$ random variables as well. 
\end{definition}

Throughout this section, we assume the dimension $p$ is fixed.  We are
interested in establishing a pivotal quantity for $T_n = T(D_n)$ like
\eqref{eq:pivot:uni:rand} in Example \ref{example2} where $T_n$ is the sample
mean after the randomized ``file drawer effect''.  It turns out we have an
exact pivotal quantity if $T_n$ is normally distributed. To lighten notation,
we suppress the script $n$ in the following lemma, which is a finite sample
result valid for any $n$.  We prove the lemma in Section \ref{sec:proof}.

\begin{lemma}
    \label{lem:exact}
    If the statistic $T$ is normally distributed from $N(\mu, \frac{\Sigma}{n})$
    and the model $M$ is selected by randomized selection $\mQ^*(T, \omega)$, where $\omega \sim \Q$.
    Then for any contrast $\eta$, which could depend on the outcome of selection $\mQ^*$, we have 
    \begin{equation}
        \label{eq:pivot}
        P(\stat; \eta^T\mu,\Sigma) = \frac{\int_{\eta^T\stat}^{\infty} \Q(t; V_{\eta}) \cdot \exp\left( -n(t-\eta^T\mu)^2 / 2\sigma_{\eta}^2\right)\; dt}{\int_{-\infty}^{\infty} \Q(t; V_{\eta}) \cdot \exp( -n(t-\eta^T\mu)^2 / 2\sigma_{\eta}^2)\; dt} \overset{\F^*}{\sim} \Unif(0,1)
    \end{equation}
where 
    $$
    \begin{aligned}
        \sigma_{\eta}^2 &= \eta^T \Sigma \eta, \quad V_{\eta} = \left(I - \frac{1}{\sigma^2_{\eta}} \Sigma \eta\eta^T\right) T \\
        \Q(t, V_{\eta}) &= \Q\left(\left\{\omega : M \in \mQ^*\bigg(t\cdot \Sigma\eta/\sigma_{\eta}^2 + V_{\eta}, \omega\bigg)\right\}\right).
    \end{aligned}
    $$
\end{lemma}

\begin{remark}
\label{rmk:nullstat}
In selected models $M_{sel,E}$, the selection is often made not only
based on $(T, \omega)$, but also other statistic of the data, which
we call the null statistic $N$. Thus the selection event should be
expressed as $\{M \in \mQ^*\big((T, N), \omega\big)\}$. To make notation simpler, we exclude
such possibilities. But a slightly modified pivot where we replace $\Q(t; V_{\eta})$
with $\Q(t;V_{\eta}, N)$ in \eqref{eq:pivot} and integrate over $N$, is still
$\Unif(0,1)$ distributed.
\end{remark}

Note that Lemma \ref{lem:exact} provides a valid pivotal quantity for any randomized selection procedure $\mQ^*$
and any randomization noise $\Q$ provided that $T$ is normally distributed. In fact, Lemma \ref{lem:exact} does not
require $T$ to be a linearizable statistic. In some sense, the lemma is a reformulation (after rescaling)
of the selective tests constructed in the linear regression model with additive noises (see Section \ref{sec:linear:additive}). 
For example, in the selected model $M_{sel, E}$, to test the hypothesis $H_{0j}: \beta_j = 0, ~ j \in E$, we consider the law
\eqref{eq:poly:inference:law}. After introducing the null statistic $N = P_E^{\perp} y$, 
the pivot in \eqref{eq:pivot} is in fact the CDF transform of this law, taking 
$
T = P_E y, ~ \Sigma = n\sigma^2 P_E, 
$
and the selection event $\{M \in \mQ^*((T,N), \omega)\}$ to be the affine selection event defined in \eqref{eq:poly:inference:law}. 
With simple calculation, it is easy to see $V_{\eta} = (P_E - \norm{\eta}^{-2} \eta^T\eta) y = P_{E\backslash \eta} y$,
which we condition on in both \eqref{eq:pivot} and \eqref{eq:poly:inference:law}.

Of course the pivot in \eqref{eq:pivot} is very difficult to compute explicitly, and we need to use sampling
schemes like in \eqref{eq:poly:inference:law}. But in a nutshell, $P(T; \eta^T\mu, \Sigma)$ is simply a CDF
transform of the law
\begin{equation}
\label{eq:pivot:law}
\eta^T T \mid V_{\eta}, ~M \in \mQ^*(T, \omega), \quad (T, \omega) \sim N\left(\mu, \frac{\Sigma}{n}\right)\times \Q.
\end{equation}
After introducing the null statistic, 
Lemma \ref{lem:exact} is agnostic to the selected model $M_{sel,E}$, where $\mu = X_E\beta_E$ or the
saturated model $M_{sat}$, where the parameter is simply $\mu$. The nuances between the two models
in terms of sampling is that the saturated model condition on $N$ (treating it as part of $V_{\eta}$), but
selected model integrate over $N$. 

Lemma \ref{lem:exact} is written with $T$ implicitly being the approximate average of $n$ i.i.d variables,
hence the distribution $N(\mu, \frac{\Sigma}{n})$.
Linearizable statistics are of particular interest as they converge to $N(\mu, \frac{\Sigma}{n})$ due to
central limit theorem. In the following, we seek to establish conditions under which the pivot
$P(T; \mu, \Sigma)$ will be asymptotically $\Unif(0,1)$.  

\subsection{Selective central limit theorem}

In other work on asymptotics of selective inference \cite{tian2015asymptotics,tibshirani2015uniform}, 
the setup considered is usually the saturated model $M_{sat}$.
These works considered asymptotics of selective inference marginalized over the range of 
$\mQ^*$. In contrast, we consider the convergence for any particular
selected model $M_n$, under the conditional law of the selection event
$\{M_n \in \mQ^*_n\}$. Specifically, we allow weak convergence of the pivot in \eqref{eq:pivot}
in the sequence of selected models $(M_n)_{n \geq 1}$. As explained above, selected models integrate over
the null statistics while saturated models condition on those, thus the selective
tests should have more power provided that the selected model is believable. 
In the saturated model, our result provides a
finer measure of convergence than in \cite{tian2015asymptotics}. On the
other hand, \cite{tian2015asymptotics} allows high-dimensional setting
in some cases while we consider fixed dimension $p$.

Similar to the asymptotic setting in Section \ref{sec:nonparametric_setup},
we consider the convergence of $P(T_n; \eta^T \mu_n, \Sigma_n)$ under a sequence
of models $(M_n)_{n \geq 1}$ selected by a sequence of selection procedures $(\mQ^*_n)_{n \geq 1}$.
$(T_n)_{n \geq 1}$ is a sequence of linearizable statistics defined in Definition \ref{def:decomposable},
with asymptotic mean $\mu_n$ and asymptotic covariance matrix $\frac{\Sigma_n}{n}$.

It turns out that in this setting, the selective likelihood ratio $\ell_{\F_n}$ again
plays an important role in the convergence of the pivot. Recall that with randomized selection
$\mQ^*(T_n, \omega)$, the selective likelihood is
\begin{equation}
\begin{aligned}
&\ell_{\F_n}(T_n; M_n) = \frac{\modelweight(T_n; M_n)}{\Esubarg{\F_n}{\modelweight(T_n; M_n)}}, \\
& \modelweight(T_n; M_n) = \Q\left(\left\{\omega: M_n \in \mQ_n^*(T_n, \omega) \right\}\right)
\end{aligned}
\end{equation}

It will be convenient to rewrite the likelihood ratio 
in terms of the normalized vector $Z_n=\sqrt{n}(T_n - \mu_n)$
\begin{equation}
\bar{\ell}_{\F_n}(Z_n) = \ell_{\F}(n^{-1/2} Z_n + \mu_n).
\end{equation}
as well as the pivot \eqref{eq:pivot}
\begin{equation}
\bar{P}_{\F_n}(Z_n) = P(n^{-1/2} Z_n + \mu_n; \eta_n^T\mu_n,\Sigma_n).
\end{equation}

Our approach is basically a comparison of how the pivot will behave
under $\F_n$ and its Gaussian counterpart $\Phi_n = N(\mu(\F_n),
\Sigma(\F_n))$.  Specifically, it is a modification of the proof of Theorem 1.1
of  \cite{chatterjee2005simple}, modified to allow for the fact the derivatives
of the pivot and the likelihood are not required to be
uniformly bounded. Given a norm $\Omega$ on $\real^p$, define 
\begin{equation}
\label{eq:sobolev}
\lambda^{\Omega}_r(f) = \sup_{s \in \real^{p}} \left\{ \big\|\partial^k f(s)\big\|^{r/k}\exp(-r\Omega(s)): 1 \leq k \leq r \right\},
\end{equation}
where $\partial^k$ denotes the k-fold differentiation with respect to the $p$-dimensional vector $s$, $\| \cdot \|$ denotes element wise maximum. 

Now we state our selective central limit theorem, which we prove in Section \ref{sec:proof}.
\begin{theorem}[Selective central limit theorem]
\label{thm:weak:conv}
Suppose the statistics $T_n = T(D_n)$ are linearizable statistics according to Definition \ref{def:decomposable}.
We also assume the norms $\Omega: \real^{p} \rightarrow \real$ are such that for each $f \in \{\bar{P}_n, \bar{\ell}_{\F_n}, \bar{\ell}_{\Phi_n}\}$,
it satisfies
\begin{equation}
    \label{eq:derivatives}
\sup_{\F_n \in M_n} \lambda^{\Omega}_3(f) \leq  C_1.
\end{equation}
Moreover, assume $\xi_{i, n}$ has uniformly bounded moment generating function in some neighbourhood of $0$.
Namely, $\exists a > 0$, such that 
\begin{equation}
    \label{eq:exp_moment}
\sup_{n \geq 1} \sup_{\F_n \in M_n} \Ee_{\F_n}(\exp(a\|\xi_{i,n} - \mu(\F_n) \|_1)) \leq C_2.
\end{equation}
Furthermore, we assume
\begin{equation}
\label{eq:rare}
\limsup_n n^{1/2}\cdot \frac{\Pp_{(\F_n \times \Q)}[M_n \in \mQ_n^*] - \Pp_{(\Phi_n \times \Q)}[M_n \in \mQ_n^*]}{\Pp_{(\Phi_n \times \Q)}[M_n \in \mQ_n^*]} \leq C_3.
\end{equation}

Then, for any $g$ with uniformly bounded derivatives up to third order
\begin{equation}
\label{eq:conv}
\left| \Ee_{\F_n^*}\left[g\big(P(T_n)\big)\right] - \int_{0}^1 g(x) dx\right| \leq 
n^{-1/2} K(g, C_1, C_2, C_3, p), \quad n \geq n_0
\end{equation}
where $K$ depends only on the bounds on the
derivatives of $g$, the constants  $C_1, C_2, C_3$ and the dimension $p$. 
Thus the convergence is uniform in $(M_n)_{n \geq 1}$ for models 
satisfying \eqref{eq:derivatives}, \eqref{eq:exp_moment} and \eqref{eq:rare}. 
\end{theorem}

Theorem \ref{thm:weak:conv} provides a finite sample bound on the convergence of the
pivot $P(T_n)$. Since we allow $g$ to be functions with uniformly bounded derivatives
up to the third order, \eqref{eq:conv} implies convergence of $P(T_n)$ to $\Unif(0,1)$
under $\F_n^*$. In the following examples, we show how to verify conditions \eqref{eq:derivatives},
\eqref{eq:exp_moment} and \eqref{eq:rare}.

\subsection{Revisit the ``file drawer problem''}
\label{sec:mean:CLT}

In Examples \ref{example1} and \ref{example2}, we considered only reporting an
interval or a $p$-value about $\mu_n$ when $n^{1/2}\bar{X}_n > 2$ or
$n^{1/2}\bar{X}_n + \omega > 2$. This is an example where we do not really
select a model, but rather select only a proportion of the data to report.  The
selective distribution simply refers to the law of the reported sample means,
which pass the threshold.

The data we observe is $D_n = (X_{1,n,\dots, X_{n,n}})$ with the linearizable
statistic $T_n$ simply being the sample mean $\bar{X}_n$.
Example \ref{example1} corresponds to the degenerate randomization of adding
0 to $\bar{X}_n$. Work of \cite{tian2015asymptotics} show that in order
for the corresponding pivot to converge weakly we can take, for $\Delta < 0$ fixed
\begin{equation}
\label{eq:example1:model}
M_n = \left\{F: \E_F[\bar{X}_n] > n^{-1/2} \Delta, \E_F[X_{i,n}^3] < \infty \right\}.
\end{equation}
That is, $\bar{X}_n$ will satisfy a selective CLT when
the population mean is not too negative. 

On the other hand, in Example \ref{example2}, the pivot in \eqref{eq:pivot} is of the form, 
\begin{equation}
\label{eq:pivot:example2}
P(\bar{X}_{n}) = \frac{\int_{\bar{X}_{n}}^{\infty} \bar{G}(2  - \sqrt{n}t) e^{-n(t-\mu_n)^2/2}  \; dz}{\int_{-\infty}^{\infty} \bar{G}(2  - \sqrt{n}t) e^{-n(t-\mu_n)^2/2} \; dz},
\end{equation}
and likelihood $\ell_{\F_n}(\bar{X}_n)$ is defined in \eqref{eq:likelihood2}. 

When $G$ is the Logistic noise, then condition \eqref{eq:derivatives} and \eqref{eq:rare}
can be verified. Formally, we have the following lemma whose proof we defer to the appendix,
\begin{lemma} 
\label{lem:file:drawer}
If $G = \Logistic(\kappa)$, with $\kappa$ being the scale parameter, then if centered $X_{i,n}$'s
have moment generating functions in the neighbourhood of zero,
then the pivot $P(\bar{X}_n)$ is asymptotically $\Unif(0,1)$.
\end{lemma}

In other words, with Logistic randomization noise,
we can take the sequence of models to be
\begin{equation}
\label{eq:example2:model}
M_n = \left\{\F_n: \E_{\F_n}\left[\exp\big(a\abv{X_{1,n} - \mu_n}\big)\right] < \infty \right\}, \text{ for some } a>0.
\end{equation}
Requiring exponential moments is stricter than the third moment condition in 
\eqref{eq:example1:model}, but
we would have a stronger conclusion, namely weak convergence uniformly
over all $\mu_n$'s. 

\subsection{Two-sample median problem}
\label{sec:median}

In the two-sample median problem, 
we have two treatment groups from which we take measurements, $x_{1i} \iid \F_1$ and $x_{2i} \iid \F_2$;
for simplicity of notation, we assume we observe $n$ samples from each group, and drop $n$ in the subscript. 
We will report the bigger median from this group in the non-randomized setting. Exact
formulation of randomized selection will be discussed below.

Suppose our underlying distribution is $\F = \F_1 \times \F_2$. Let $\mu = (\mu_1, \mu_2)$ is the population median of the two groups,
and $T = (T_1, T_2)$ is the sample median. The well-known result by \cite{sample_median} states that the sample
median is a linearizable statistic for the median when the CDF of the distribution $F$ has positive density $f$,
and $f'$ is bounded in a neighbourhood of the population median $m$. Formally, if $x_i \iid F$, then the sample
median
\begin{equation}
\label{eq:median:decomp}
T(x_1, \dots, x_n) = m + \frac{1}{n} \sum_{i=1}^n \frac{\mathbf{1}\{x_i > m\} - 1/2}{F'(m)} + R_n, 
\end{equation}
with $R = O(n^{-3/4} \log n)$ with probability $1$.

Our (randomized) selection algorithm $\mQ^*$ reports
\[
    \begin{cases}
        P(T;\mu_1,\Sigma), & \text{if } T_1 > T_2 + n^{-1/2}\omega \\
        P(T;\mu_2,\Sigma), & \text{if } T_1 \leq T_2 + n^{-1/2}\omega,
    \end{cases}
\]
where $\omega \sim \Q$ and $\Sigma = \diag(\frac{1}{4}f_1(\mu_1)^{-2}, \frac{1}{4}f_2(\mu_2)^{-2})$ is a diagonal matrix.
$f_1$, $f_2$ are the densities of $\F_1$ and $\F_2$.
Without loss of generality, we suppose $M_1$ is selected, i.e. the first group is the ``best'' group.

We choose the randomization noise $\Q$ to be a $\text{Logistic}(\kappa)$ with mean $0$ and $\kappa$ is the scale, and let $G_{\kappa}$
be the CDF.
The resulting pivot for $\mu_1$ is
$$
P(T; \mu_1,\Sigma) = \frac{\int_{T_1}^{\infty}  G_{\kappa}(\sqrt{n}t - \sqrt{n}T_2)  \cdot \exp(-n(t-\mu_1)^2 / 2\sigma_1^2)\; dt}{\int_{-\infty}^{\infty} G_{\kappa}(\sqrt{n}t - \sqrt{n}T_2) \cdot \exp(-n(t-\mu_1)^2 / 2\sigma_1^2)\; dt}, ~ \sigma_1^2 = \frac{1}{4f_1(\mu_1)^2}.
$$
This pivot strikes a similarity with the pivot in \eqref{eq:pivot:example2} for 
Example \ref{example2} with the truncation threshold $2$ being replaced by $\sqrt{n}T_2$ 
and plugging in the appropriate means and variances of the medians. A result similar
to Lemma \ref{lem:file:drawer} can be established, which ensures convergence of the pivot
uniformly for any underlying medians $(\mu_1, \mu_2)$.

In order to construct the above pivot, we need knowledge of the variance
$\sigma^2_1$. Without selection, there are natural estimates of this
variance. One may ask, how will inference be affected if we plug
this estimate into our pivot? We revisit this question in Section \ref{sec:plugin:variance}.

\subsection{Affine selection events}

In this section, we discuss the special case of affine selection events (regions). 
This combined with the asymptotic result in Theorem \ref{thm:weak:conv} applies to
more general settings. In particular, it allows us to approximate non-affine regions. 
For a concrete example, see Section \ref{sec:randomized_logistic}. 

We drop the subscript $n$ where possible to simplify notations.
Suppose for our model $M$, the selection is based on $(T,\omega)$, and the selection event $\{M \in \mQ^*\}$ can be described as 
$$
\{\sqrt{n}A_M T + \omega \in K_M\},
$$ 
where the affine matrix $A_M \in \real^{d \times p}$ and $K_M$ is a region in $\real^d$.
Many examples of non-randomized selective inference can be expressed in
this way (c.f. \cite{lee_lasso,spacings,exact_screening,sequential_selection}).
In this section, we provide conditions under which Theorem \ref{thm:weak:conv}
can be applied.

We again normalize $T$ to be $Z = \sqrt{n}(T - \mu)$, then
the selection event can be rewritten as 
\begin{equation}
    \label{eq:selection_rr}
    \{A_M(Z+\Delta) + \omega \in K_M\},
\end{equation}
where $\sqrt{n}\mu = \Delta$, $Z$ converges to $N(0, \Sigma)$.

Suppose $\omega \sim \Q$, which has distribution function $G$. Then we introduce some conditions on the selection region $K_M$ 
and the added noise distribution $G$,
\begin{description}
\item{\bf Lower bound:}
We assume there is some norm $h$, such that 
$$
\int_{K_M - \theta} G(dw) \geq C^- \exp\left[-\inf_{w \in K_M -\theta} h(w)\right], ~ \forall \theta \in \real^d.
$$
\item{\bf Smoothness:}
Suppose $G$ has density $g$, we assume the first 3 derivatives of $g$ are integrable,
$$
\int_{\real^d}\|\partial ^j g(w)\| dw \leq C_j, j=0,1,2,3
$$
where the norm on the left-hand side is the maximum element-wise of the partial derivatives.
\end{description}

The above two conditions essentially require $G$ to be differentiable and have heavier tails than (or equal to) exponential tails.
In fact we prove that the lower bound and smoothness conditions ensure that \eqref{eq:derivatives} are satisfied under the
{\em local alternatives} introduced below.

\begin{definition}[Local alternatives]
    For the sequence of selected model $(M_n)_{n \geq 1}$, we define the local alternatives of radius of $B$ 
    to be the set all sequences $(\mu_n)_{n \geq 1}$, such that  
    $$
    d_h(0, K_{M_n} - A_{M_n} \Delta) \leq B, \Delta = \sqrt{n}\mu_n
    $$
    where $d_h(\cdot,\cdot)$ is the distance induced by the norm $h$.
\end{definition}
The notion of local alternatives is natural in the asymptotic setting as we
expect even a small effect size will be more prominent when we collect more and
more data. 

Formally, we have the following lemma, whose proof is deferred to the appendix.

\begin{lemma}
    \label{lem:asymptotic}
    Suppose $G$, $K_M$ satisfy the lower bound and smoothness conditions, then condition \eqref{eq:derivatives}
    are satisfied under the local alternatives.
    
\end{lemma}

Now, we are left to verify conditions \eqref{eq:exp_moment} and \eqref{eq:rare}. Condition \eqref{eq:exp_moment}
is essentially a moment condition on the centered statistics $\xi_{i,n} - \mu_n$, which we have to assume.
Condition \eqref{eq:rare} can be verified using the well known results in multivariate CLT (see \cite{multivariate_clt}).
To be rigorous, we state the following lemma, which we also prove in the appendix.

\begin{lemma} 
\label{lem:berry}
If $\F_n$ is such that the centered statistics $\xi_{i,n} - \mu_n$ have finite third moments, then under
the local alternatives, condition \eqref{eq:rare} is satisfied.
\end{lemma} 

To summarize, Lemma \ref{lem:asymptotic} and Lemma \ref{lem:berry} state that 
if $G$ has integrable derivatives and exponential tails,
then the pivot in \eqref{eq:pivot} converges to $\Unif(0,1)$ uniformly for $\F_n^*$ so long as $\F_n$'s
are such that $\xi_{i,n} - \mu_n$ have exponential moments in a neighbourhood of $0$.

Unlike the sample mean and sample median examples, the pivot is difficult to compute explicitly in
this case. However, as we discuss in the beginning of Section \ref{sec:CLT}, the pivot is essentially
the CDF transform of the conditional law \eqref{eq:pivot:law}, which we can sample from. 
As discussed above, we can just take $\omega$ to be from a Logistic distribution.

Now we apply the above theory to logistic regression.

\subsubsection{Example: randomized logistic lasso}
\label{sec:randomized_logistic}

Suppose we observe independent samples, $d_i = (y_i, x_i) \iid \F$, where $y_i$'s are binary observations
and $x_i \in \real^p$. The ordinary logistic regression solves the following problem,
\begin{equation}
    \label{eq:logistic}
    \begin{aligned}
    \bar{\beta} &= \text{argmin}_{\beta \in \real^p} \ell(\beta) \\
    &= \text{argmin}_{\beta \in \real^p} -\left[\sum_{i=1}^n y_i \log \pi(x_i \beta) + (1-y_i) \log(1-\pi(x_i \beta))\right],
    \end{aligned}
\end{equation}
where $\pi(x) = \exp(x) / (1+\exp(x))$.
This is a nonparametric setting as we do not assume any parametric structure for $\F$. 

The randomized logistic lasso adds an $\ell_1$ penalty, a randomization term and a small quadratic term, 
\begin{equation}
    \label{eq:randomized_logistic}
    \hat{\beta} = \text{argmin}_{\beta \in \mathbb{R}^p} \frac{1}{\sqrt{n}}\ell(\beta) + \omega^T\beta + \|\Lambda\beta\|_1 + \frac{1}{2\sqrt{n}} \|\beta\|_2^2,
\end{equation}
where $\omega_j \iid \text{Logistic}(\kappa)$ is the perturbation to the gradient and $\Lambda$ is a diagonal matrix which introduces
(possibly) unequal feature weights, $\kappa$ controls the amount of randomization added. 
The addition of the quadratic term ensures that \eqref{eq:randomized_logistic} is strictly convex, thus has a unique solution. 
A similar formulation for linear regression has been proposed in \cite{stability_selection}.

Selective inference in this setting has not been considered before. Without the Gaussian assumptions
\cite{exact_lasso} does not apply. The parametric setting of this problem has been discussed in \cite{optimal_inference},
but computation of the selective tests are mostly infeasible for general $X$. Finally,
the asymptotic result by \cite{tian2015asymptotics} does not apply here as the framework require exactly affine selection
regions, which is not the case in this setting. 

%Compared with the ordinary logistic regression, the randomized logistic lasso enforces sparsity and the randomization
%preserves more information for inference.
Suppose the solution to \eqref{eq:randomized_logistic} has nonzero entry set $E$, then our target of inference $\beta_E^*$,
the unique population minimizer which satisfies
\begin{equation}
\label{eq:population:param}
\Ee_{\F} [X_E^T(y - \pi(X_E\beta_E^*))] = 0.
\end{equation}
Note that a parametric model $y_i | x_i \sim \text{Bernoulli}(\pi(x_{i,E}\beta_E^*))$ with independently
sampled $x_i$'s will have $\beta_E^*$ satisfying \eqref{eq:population:param}. But we by no means assume such
an underlying distribution. Rather, for any well-behaved distribution $\F$, $\beta_E^*$ can be thought of
of a statistical functional of the underlying distribution $\F$, depending on the outcome of selection $E$.

Selective inference in this setting is carried out
conditioned on $(E, s_E)$, the active set and its signs. We first introduce the following notations,
$$
\begin{aligned}
&\pi_E(\beta_E) = \frac{\exp(X_E\beta_E)}{1 + \exp(X_E\beta_E)}, \quad 
W_E(\beta_E) = \text{diag}(\pi_E(\beta_E)(1-\pi_E(\beta_E))), \\
&Q_E(\beta_E) = \frac{1}{n}X_E^TW_E(\beta_E)X_E, \quad 
C_E(\beta_E) = \frac{1}{n}X_{-E}^TW_E(\beta_E)X_E, \\
&D_E(\beta_E) = C_E(\beta_E) Q_E^{-1}(\beta_E)
\end{aligned}
$$
where $X$ is the feature matrix, and $X_E$, $X_{-E}$ is the columns
corresponding to the active set and inactive set respectively.
By law of large numbers, we have 
\begin{equation}
    \label{eq:matrix_approx}
\begin{aligned}
&Q_E(\beta_E^*) \overset{p}{\rightarrow} \Ee_{\F} Q_E(\beta_E^*) \overset{def}{=} Q, \quad
C_E(\beta_E^*) \overset{p}{\rightarrow} \Ee_{\F} C_E(\beta_E^*) \overset{def}{=} C,\\
&D_E(\beta_E^*) \overset{p}{\rightarrow} CQ^{-1}\overset{def}{=} D.
\end{aligned}
\end{equation}

Now we introduce our linearizable statistics and show that the conditioning event $(E, s_E)$
can be expressed as affine regions of these statistics.

\begin{lemma} 
\label{lem:logistic}
Suppose $E$ is the active set of the solution of \eqref{eq:randomized_logistic}, and we denote
$$
\bbeta_E = \argmin_{\beta_E \in \real^{E}} -\left[\sum_{i=1}^n y_i \log \pi(x_{i,E} \beta_E) + (1-y_i) \log(1-\pi(x_{i,E} \beta_E))\right]
$$
as the unpenalized MLE restricted to the selected variables $E$.

The following statistic $T$ is linearizable with asymptotic mean $(\beta_E^*, \rho)$ and variance $\Sigma/n$,
$$
T = \begin{pmatrix} 
\bbeta_E \\ \frac{1}{n} X_{-E}^T \left[y - \pi_E(\bbeta_E)\right]
\end{pmatrix} + R,
$$
where $R = o_p(n^{-1/2})$ is a small residual, and $\rho = \Earg{x_{i,-E}^T (y_i - \pi(x_{i,E}\beta_E^*))}$.
Moreover, the selection event $\{\hat{E}, z_{\hat{E}} = (E, s_E)\}$
can be characterized as the affine region $\{\sqrt{n}A_M T + B_M \omega \leq b_M\}$, where
$$
A_M = 
\begin{pmatrix}
    -S_E & 0 \\
    0 & I_{-E} \\
    0 & -I_{-E}
\end{pmatrix}, \quad
B_M = 
\begin{pmatrix}
S_EQ^{-1} & 0 \\
D & -I_{-E} \\
-D &  I_{-E}
\end{pmatrix},\quad
b_M =
\begin{pmatrix}
- S_E Q^{-1}\Lambda_E s_E\\
\lambda_{-E} - D \Lambda_E s_E\\
\lambda_{-E} + D \Lambda_E s_E\\
\end{pmatrix},
$$
where $I_{-E}$ denotes the identity matrix of $n-|E|$ dimensions and $\Lambda_E$, $\Lambda_{-E}$ denote
the active block and the inactive block of $\Lambda$ respectively and $\lambda$ is the
diagonal elements of $\Lambda$, $S_E = \diag(s_E)$.

\end{lemma} 
The proof of this lemma is also deferred to the appendix.

Thus using Lemma \ref{lem:asymptotic} and Lemma \ref{lem:berry}, we can conclude under local alternatives,
the pivot \eqref{eq:pivot} converges to $\Unif(0,1)$.
To test $H_{0j}:\beta_j^* = 0$, we take $\eta = e_j$, and sample 
$$
\eta^T T \mid V_{\eta}, ~ \sqrt{n}A_M T + B_M \omega \leq b_M, \quad (T, \omega) \sim N\left(\begin{pmatrix} \beta^*_E \\ \rho \end{pmatrix}, \frac{\Sigma}{n}\right) \times G,
$$
where $\rho = \Earg{x_{i,-E}^T (y_i - \pi(x_{i,E}\beta_E^*))}$. Since $\rho$ is the nuisance
parameters for testing $H_{0j},~ j \in E$, the conditional law above will not depend on its value.
A hit-and-run algorithm for sampling this law can be implemented. Moreover, recent development by
\cite{magic, harris2016selective} propose more general and efficient sampling schemes for this law.
For details, see for example Chapter 3.2 \cite{magic} where the sampling scheme
for this very example is considered and simulation results are provided.

In Lemma \ref{lem:logistic}, we assume the covariance matrix $\Sigma$ is known. % the Fisher Information matrix,
In applications, we can bootstrap it. But is it valid to plug in the bootstrap estimate of $\Sigma$?

\subsection{Plugging in variance estimates}

\label{sec:plugin:variance}

In Section \ref{sec:median} we derived quantities that were asymptotically pivotal for the best median, up to an 
unknown variance.
In the sample median case, by \eqref{eq:median:decomp}, the variance of the sample median is approximately $[4nf(m)^2]^{-1}$, where
$f(m)$ is the PDF evaluated at the median $m$. A simple consistent estimator for $f(m)$ is to take $1/2 \pm \frac{1}{\sqrt{n}}$ quantiles $a_n$ and
$b_n$, then
\begin{equation}
    \label{eq:var_estimate}
f(m) \approx \frac{2}{\sqrt{n}(b_n-a_n)}
\end{equation}
is consistent for $f(m)$ based on which we get a consistent estimator for $\sigma_1^2$. 

More generally, computing the pivot \eqref{eq:pivot} requires knowledge of $\Sigma$. 
In practice, we usually do not have prior knowledge of the variance $\Sigma$ and need a consistent estimate for $\Sigma$. 
We might use a bootstrap or jackknife estimator. When $p$ is fixed, the bootstrap 
estimator is consistent and thus we get a consistent estimator $\hat{\Sigma}$.
Lemma \ref{lem:transfer} states that under moment conditions on the likelihood, $\hat{\Sigma}$
will be consistent for $\Sigma$ under $\F_n^*$ as well, justifying the plug-in estimator of $\Sigma$.

Figure \ref{fig:asymptotics} is some simulation results for the two-sample medians problem. In each case, we take the sample size
for each treatment group to be $500$, and generate the noise from a skewed distribution $N(0,1) + 0.5\Exp(1)$. We standardize it 
such that the noise has median $0$ under the null hypothesis. We use additive logistic noise with scale $\kappa=0.8$ for randomization. 
The better group is decided using the randomized sample median, and selective inference is carried out.
In Figure \ref{fig:pivot}, the pivot with plugin variance estimate $\hsigma$ 
in \eqref{eq:var_estimate} is plotted under both the null hypothesis
$H_0:\mu_{better} = 0$ and the $H_A:\mu_{better} > \frac{1}{\sqrt{n}}$. 
The pivot has reasonable power even for identifying local alternatives.
The pivot is almost exactly $\Unif(0,1)$ under the null hypothesis
with the sample size $n=500$. In fact, it is very close at a relatively small sample size $n=50$ justifying the application of
asymptotics in the nonparametric setting. Figure \ref{fig:hist} further illustrates the difference in the unselective v.s. 
selective distribution and its convergence to its theoretical limit. We see
that there is a clear shift in selective distribution that calls for adjustment
for the selection.  For sample size $n=500$, the empirical selective
distribution converges to our theoretical distribution. 

\begin{figure} 
    \centering
    \begin{subfigure}[t]{0.45\textwidth}
                \includegraphics[width=\textwidth]{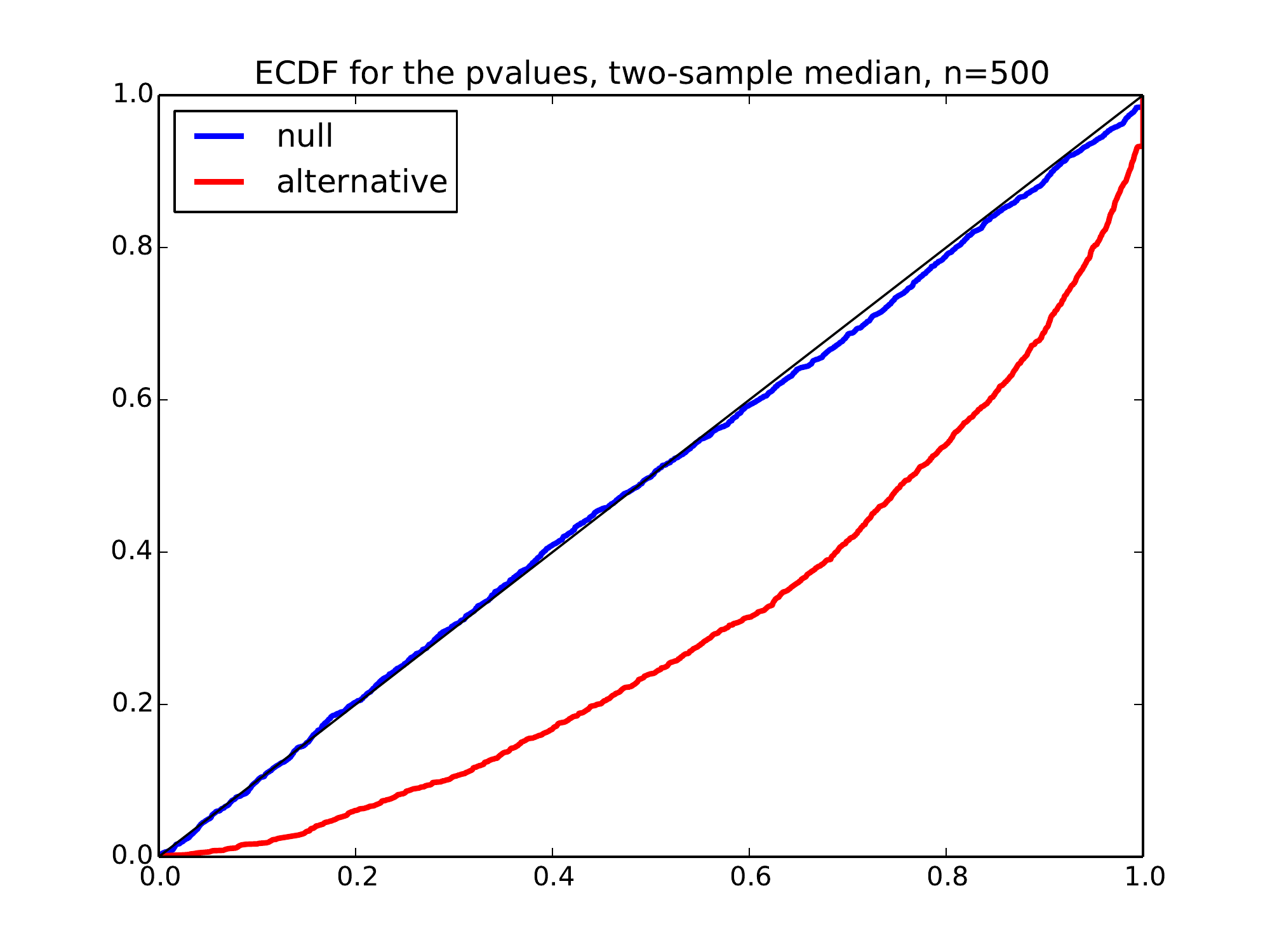}
                \caption{Null and alternative pivot for the ``better'' median}
                \label{fig:pivot}
        \end{subfigure}%
        ~ 
        \begin{subfigure}[t]{0.45\textwidth}
                \includegraphics[width=\textwidth]{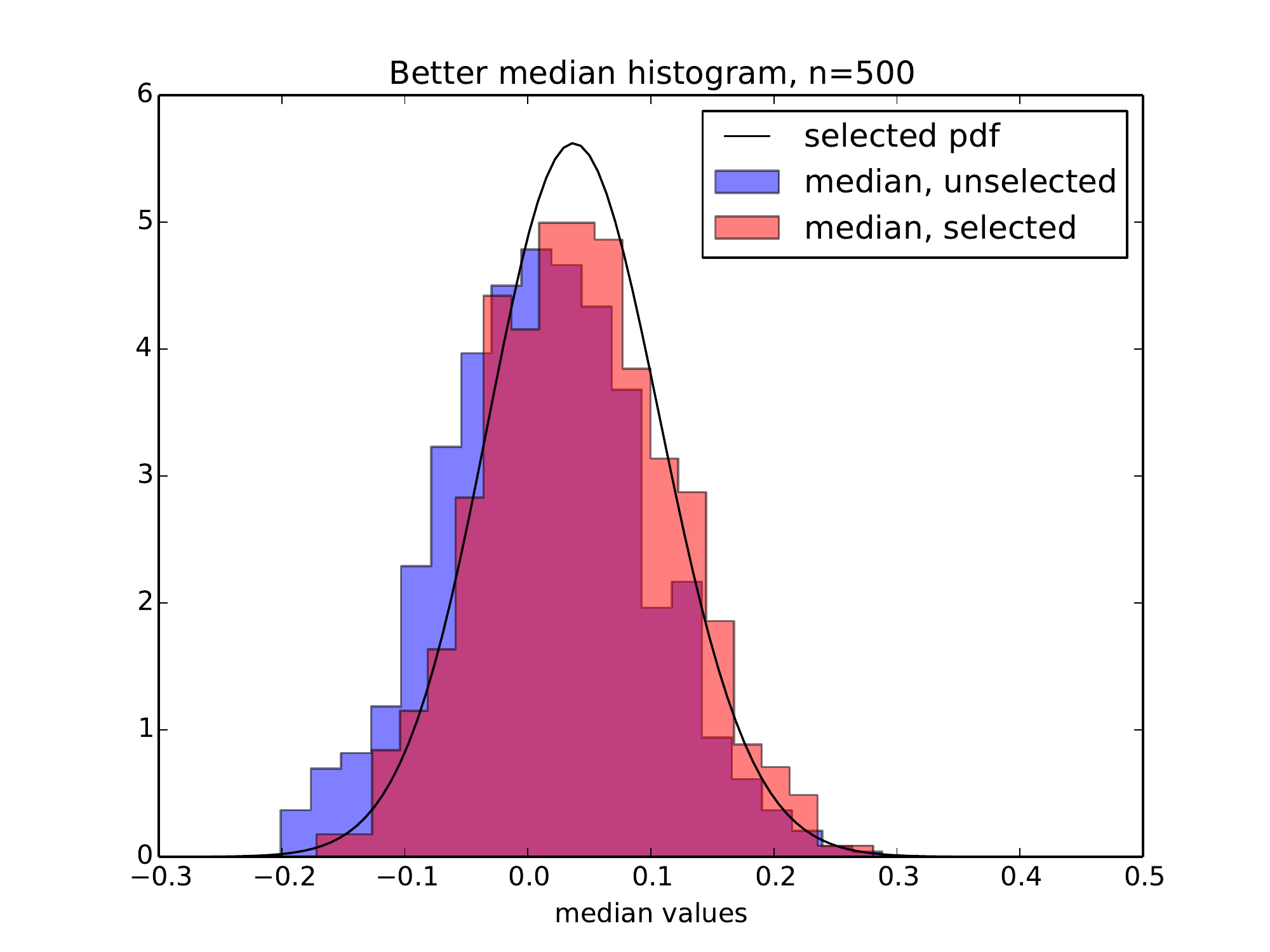}
                \caption{Selective v.s. unselective distribution, and theoretical PDF}
                \label{fig:hist}
        \end{subfigure}
        \caption{Asymptotic distribution of the median for the selected group}\label{fig:asymptotics}
\end{figure}

\section{Multiple Randomizations of the Data} 
\label{sec:multiple}

Most of the examples above focus on a single randomization $\omega$ on the data,
which we use for model selection. We naturally want to extend it to multiple
randomizations, and multiple randomized selections, which will collectively
suggest a model for inference. In this section, we allow multiple randomizations
in a possibly sequential fashion and discuss how inference can be carried out. 
%In our earlier examples on regression, we considered selection procedures such
%as the LASSO based on a randomized response in which Gaussian noise is added
%before choosing a model. Following this, we used an existing selection
%procedure and updated the selective distribution with the selective likelihood
%ratio. In practice, researchers may want to use a more complex algorithm to
%select the model, such as choosing regularization parameter with
%cross-validation. In particular, this may require multiple randomizations of
%the response to achieve tractable sampling algorithms or perhaps more
%interpretable results. 

\subsection{Selective inference after cross-validation}

Consider the case where we first choose a regularization parameter by cross-validation,
and then fit the square-root LASSO problem \cite{sqrt_lasso} at this parameter,
\begin{equation}
\label{eq:sqrt:lasso}
\hat{\beta}_{\lambda}(y;X) = \argmin_{\beta} \|y-X\beta\|_2 + \lambda \|\beta|_1,
\end{equation}
where $\lambda$ is picked from a fixed grid $\Lambda=[\lambda_1, \dots, \lambda_k]$. 
The discussion below is not specific to selection by square-root LASSO.

The model selected by cross-validated square-root LASSO involves two steps of
selection. We denote by $y_{\CV}$ the response for selecting the randomization
parameter, and $y_{\select}$ the response vector for fitting the square-root
LASSO at the selected regularization parameter $\lambda$. Both vectors are
randomized version of the original vector $y$. Inference after cross validation
requires combining two steps of randomized selection. Consider the following procedure. 

First, we randomize $y$ to get the vector $y_{\CV}$ and $y_{\select}$
\begin{equation}
\label{eq:randomize:CV}
\begin{aligned}
y_{\text{inter}}  & \vert y, X & \sim N(y, \sigma^2_1 I) \\
y_{\text{CV}}  & \vert y_{\text{inter}}, y, X & \sim N(y_{\text{inter}}, \sigma^2_{2,\text{CV}}I) \\
y_{\text{select}}  & \vert y_{\text{inter}}, y, X & \sim N(y_{\text{inter}}, \sigma^2_{2,\text{select}} I). 
\end{aligned}
\end{equation}
Note the intermediate vector $y_{\inter}$ is introduced convenience of sampling. The above
is just one of the plausible randomization schemes.

After having randomized, we select $\lambda$ with $K$-fold cross-validation using 
$y_{\text{CV}}$:
\begin{equation}
\label{eq:selected:lambda:CV}
\hat{\lambda} = \hat{\lambda}(y_{\text{CV}}, X) = \text{argmin}_{\lambda \in \Lambda} CV_K
(y_{\text{CV}}, X,\lambda)
\end{equation}
where $CV_K(y,X,\lambda)$ is the usual $K$-fold cross-validation score with coefficients
estimated by the square-root LASSO. Alternatively, one could compute the cross-validation
score using the OLS estimators of the selected variables. 
Note that we have left implicit the randomization
that splits observations into groups. That is $\hat{\lambda}$ in \eqref{eq:selected:lambda:CV} above is a function of $(y_{\text{CV}}, X, \omega)$ where
$\omega$ is a random partition of $\{1, \dots, n\}$ into $K$ groups. When we sample $y_{\text{CV}}$ below, we redraw $\omega$ each time.

The subset of variables and signs is selected using the square-root LASSO with response
$y_{\text{select}}$:
\begin{equation}
\label{eq:selected:var:CV}
\begin{aligned}
\hat{E}(y_{\text{CV}},y_{\text{select}},X) &=
\left\{j: \hat{\beta}_{\hat{\lambda}(y_{\text{CV}},X),j} \neq 0 \right\},\\
z_{\hat{E}}(y_{\CV}, y_{\select}, X) &= \sign(\hat{\beta}_{\hat{\lambda}(y_{\CV},X)}).
\end{aligned}
\end{equation}

After seeing the selected variables $\hat{E}$, we perform inference in the selected model $M_{sel}(\hat{E})$.
Since $M_{sel}(\hat{E})$, we will still have an exponential family after selection. Per Lemma \ref{lem:exponential:family},
we sample from the following law,
$$
{\cal L} \left(X_j^Ty \bigl \vert \hat{\lambda}(y_{\text{CV}}, X)=\lambda, (\hat{E}, z_{\hat{E}}) = (E,z_E), P_{E\backslash j} y\right).
$$
The additional conditioning on the signs are for computational reasons. In fact, recent development in \cite{harris2016selective}
proposes sampling schemes that overcome these difficulties, so that we do not need to condition on this additional information.

To sample from the above law, we use a Gibbs-type sampler, which iterate over $y$, $y_{\inter}$, $y_{\CV}$ and $y_{\select}$,
conditional on the other three and the selection event. It includes the following steps.
\begin{description}
\item[Sampling $y_{\text{CV}}$] Using the conditional independence of $y_{\CV}$ and $y_{\select}$ given $y_{\inter}$, we have 
$$
{\cal L} \left(y_{\text{CV}} \bigl \vert y_{\text{inter}}, y_{\text{select}}, y, \hat{\lambda}(y_{\text{CV}},X)=\lambda \right) =   N(y_{\text{inter}}, \sigma^2_{2,\text{CV}} I) \bigl | \left\{ \hat{\lambda}(y_{\text{CV}},X)=\lambda \right\}.
$$
This is the computational bottleneck, as we do not have good description for the selection
event for cross validation. A brute-force sampling scheme will be computationally expensive,
as we need to refit the model over a grid of $\lambda$'s. Thus, we do not update $y_{\CV}$ too often. 
\item[Sampling $y_{\text{select}}$] 
The conditional independence of $y_{\select}$ and $y_{\CV}$ given $y_{\inter}$ implies,
$$
\begin{aligned}
&{\cal L} \left(y_{\text{select}} \bigl \vert y_{\text{inter}}, y_{\text{CV}}, y,
(\hat{E},z_{\hat{E}})(y_{\text{CV}},y_{\text{select}},X)=(E,z_E), 
   \hat{\lambda}(y_{\text{CV}},X)=\lambda \right) \\
 &=  N(y_{\text{inter}}, \sigma^2_{2,\text{select}} I) \bigl | \vert \left\{ (\hat{E}_{\lambda},z_{\hat{E},\lambda})(y_{\text{select}},X) = (E,z_E)\right\} 
\end{aligned}
$$ 
\cite{tian2015selective} has given an explicit description of the selection
event 
$$\{\hat{E}_{\lambda}, z_{\hat{E}, \lambda} = (E, z_E)\}.$$
Thus hit-and-run sampling provides a tractable sampling scheme.

\item[Sampling $y_{\text{inter}}$]
This is a simple step. Because the selection event is based on $y_{\CV}$ and $y_{\select}$, we have
$$
\begin{aligned}
&{\cal L}\left(y_{\text{inter}} \bigl \vert y, y_{\text{select}}, y_{\text{CV}} \right) \\
=& N\left(\frac{\frac{1}{\sigma^2_1}y + \frac{1}{\sigma^2_{2,\text{CV}}}y_{\text{CV}}+ \frac{1}{\sigma^2_{2,\text{select}}}y_{\text{select}}}{\frac{1}{\sigma^2_1} + \frac{1}{\sigma^2_{2,\text{CV}}}+ \frac{1}{\sigma^2_{2,\text{select}}}}, \left(\frac{1}{\sigma^2_1} + \frac{1}{\sigma^2_{2,\text{CV}}} + \frac{1}{\sigma^2_{2,\text{select}}} \right)^{-1} \right).
\end{aligned}
$$ 
\item[Sampling $y$]
This is also simple with our randomization scheme. Note that $y$ is conditionally independent of $y_{\select}$
and $y_{\CV}$ given $y_{\inter}$,
$$
{\cal L}\left(y \bigl \vert y_{\text{inter}}, y_{\text{CV}}, y_{\text{select}}\right) = 
{\cal L}\left(y \bigl \vert y_{\text{inter}}\right) = N \left(\frac{\frac{1}{\sigma^2} X_E\beta_E
+ \frac{1}{\sigma^2_1}y_{\text{inter}}}{\frac{1}{\sigma^2} 
+ \frac{1}{\sigma^2_1}}, \left( \frac{1}{\sigma^2} 
+ \frac{1}{\sigma^2_1} \right)^{-1} \right) 
$$
Since we condition on $P_{E\backslash j}y$, we essentially take $y$ and project
out the update on the space orthogonal to that of $X_j$.
\end{description}

A chain that iterates through the above four steps will give us samples from
the desired distribution for inference.

\subsection{Collaborative selective inference}
One of the motivations of the reusable holdout described in \cite{reusable_holdout} is that it allows a data analyst to repeatedly
query a database yet still be able to approximately estimate expectations even after asking many questions about the data. 
Another version of this model
may be that several groups wish to model the same data and then, as a consortium, decide on a final model and be able to
approximately estimate expectations in this final model. We might call this {\em collaborative selective inference}.

\newcommand{\bankS}{{\cal B}}
\newcommand{\bankP}{\mathbb{B}}

Formally, suppose each of $L$ groups has its own preferred method of model selection, encoded
as selection procedures $(\mQ_l)_{1 \leq  l \leq L}$. We assume there is a central ``data'' bank that decides
what ``data'' each group is allowed to see. We express this is as a sequence of randomization schemes $(y^*_l)_{1 \leq l \leq L}$.
Formally, this is equivalent to enlarging the probability space to $\cD \times \bankS$ with measure $\F \times \bankP$ and fixing a function
$y^*(y,\omega) = (y^*_1(y,\omega), \dots, y^*_l(y,\omega))$. It may be desirable to choose the law of $y^*|y$ so that the coordinates
are conditionally independent given $y$, though it is not necessary. 

Now suppose that
the $L$ groups choose models $\hat{M}_l^* = \mQ_l(y^*_l) \in \sigma(y^*_l)$ and convene to discuss what the best model
is $M$. For every choice of $L$ models $(M_1,\dots, M_L)$ and final model $M$, the following selective distribution can be used for valid selective inference 
\begin{equation}
\frac{d\F^*}{d\F}(y) = \frac{\bankP(\omega: \cap_{l=1}^L\mQ_l(y^*_l(y,\omega)) = M_l)}
{(\F \times \bankP)(\cap_{l=1}^L \mQ^*_l = M_l)}.
\end{equation}
When the $y^*_l$'s are conditionally independent given $y$ then it is clear that 
$$
\bankP( \cap_{l=1}^L\mQ_l(y^*(y,\omega)) = M_l) = \prod_{l=1}^L \bankP(\mQ_l(y^*_l(y,\omega))=M_l).
$$
It is possible that the consortium has beforehand decided on an algorithm that will choose a best model automatically, determined by some function ${\cal S}(M_1, \dots, M_L)$.
In this case, one should use the selective distribution
\begin{equation}
\frac{d\F^*}{d\F}(y) = \frac{\bankP(\omega : {\cal S}(M^*_1(y,\omega), \dots, M^*_L(y,\omega)) = M)}
{(\F \times \bankP)({\cal S}(M^*_1, \dots, M^*_L) = M)}
\end{equation}
When the models in question are parametric, perhaps Gaussian distributions, and the 
randomization is additive Gaussian noise the central data bank can explicitly lower bound the leftover information  by
$$
\text{Var}(y | y^*_1, \dots, y^*_L).
$$
This quantity is expressible in terms of the marginal variance of $y$ and the central data bank's noise generating distribution 
for $y^*(y,\omega)=(y+\omega_1, \dots, y+\omega_L)$. 
By maintaining a lower bound on the above quantity, the central data bank can
maintain a minimum prescribed information in the data for final estimation
and/or inference. 
In a sequential setting, where valid inference is desired at
each step, maintaining a lower bound may involve releasing noisier and noisier
versions of $y$. Sampling under this scheme seems quite difficult, and we leave
it as an area of interesting future research.

\section{Proof}
\label{sec:proof}

\subsection{Proof of Theorem \ref{thm:weak:conv}}
To prove Theorem \ref{thm:weak:conv}, we first prove the following lemma, which
might be of independent interest.

\begin{lemma}
\label{lem:ext:sourav}
Suppose $T_n$ is a linearizable statistic for $\mu_n = \mu(\F_n)$ as defined in \eqref{eq:decomposable}.
Let $Z_n = \sqrt{n}(T_n - \mu_n) \in \real^p$ and a function $f: \real^p \to \real$ with
finite $\lambda_3^{\Omega}(f)$ for some norm $\Omega$ on $\real^p$. Moreover, if $\F_n$ has
finite centered exponential moments in a neighbourhood of zero.
Then
$$
|\Esubarg{\F_n}{f(Z_n)} -\Esubarg{\Phi_n}{f(Z_n)}| \leq C(p)\lambda_3^{\Omega}(f)n^{-\frac{1}{2}}, \quad n \geq n_0
$$ 
for some $n_0 \geq 1$, where $C(p)$ is a constant only dependent on the dimension.
\end{lemma}

Lemma \ref{lem:ext:sourav} can be seen as an extension of the result by \cite{chatterjee2005simple} in the sense that
the author in \cite{chatterjee2005simple} established result for the case $\Omega = 0$. The proof is also an adaptation
of the technique in \cite{chatterjee2005simple}. 

\begin{proof}
Without loss of generality, we assume $T_n = \frac{1}{n}\sum_{i=1}^n \xi_{i,n}$ and the residual is $0$.
First, we define the normalizing operator. For any $S \in \real^{n \times p}$,
$$
\normalize(S) = \frac{1}{\sqrt{n}}\sum_{i=1}^n S[i] \in \real^p,
$$
where $S[i]$ is the $i$-th row of $S$.

We also define for any $n$ and $0 \leq k \leq n$,
$$
S_{n, k} = 
\begin{pmatrix}
    \xi'_{1,n} - \mu(\F_n),\\
    \vdots, \\
    \xi'_{k-1,n} - \mu(\F_n), \\
    \xi'_{k,n} - \mu(\F_n), \\
    \xi_{k+1,n} - \mu(\F_n),\\
    \vdots, \\
    \xi_{n,n}- \mu(\F_n)
\end{pmatrix} 
\quad
S_{n, k}^- = 
\begin{pmatrix}
    \xi'_{1,n}- \mu(\F_n),\\
    \vdots, \\
    \xi'_{k-1,n}- \mu(\F_n), \\
    0\\
    \xi_{k+1,n}- \mu(\F_n),\\
    \vdots, \\
    \xi_{n,n}- \mu(\F_n)
\end{pmatrix}
$$ 
where $\xi_{i, n} \iid \F_n$ with mean $\mu(\F_n)$ and variance $\Sigma(\F_n)$ 
and $\xi'_{i, n} \iid N(\mu(\F_n), \Sigma(\F_n))$. Let $F_{n,k}$. $F^-_{n,k}$ denote the 
distribution of $S_{n,k}$ and $S^-_{n,k}$'s respectively. Note $F_{n,k}$ and
$F_{n,k}^-$ are determined by $\F_n$. For simplicity of notation we only distinguish
the two distributions by $F_{n,k}$ and $F^-_{n,k}$, avoiding verbose notations of $S$,
e.g. $\Ee_{F_{n,k}}[S] = \Ee_{F_{n,k}}[S_{n,k}]$. It is then easy to see $Z_n = \normalize(S_{n,0})$.

Now by telescoping:
    $$
\begin{aligned}
    &\big|\Ee_{\F_n}[f] - \Ee_{\Phi_n}[ f] \big|  
    =\big|\Ee_{F_{n,0}}[f \circ \normalize(S)] - \Ee_{F_{n,n}}[f \circ \normalize(S)] \big| \\
\leq& \sum_{i=1}^n \big| \Ee_{F_{n,i-1}}[f\circ \normalize(S)] - \Ee_{F_{n,i}^-}[f\circ \normalize(S)] 
+
\Ee_{F_{n,i}^-}[f\circ \normalize(S)] - \Ee_{F_{n,i}}[f\circ \normalize(S)] \big|.
\end{aligned}
    $$

Let $\partial_i$ be the derivative with respect to the $i$-th row $S[i]$. Using Taylor's expansion
at $S_{n,k}^-$, we have
    $$
    \begin{aligned}
 &\Ee_{F_{n,i}}[f\circ\normalize] - \Ee_{F_{n,i}^-}[f\circ\normalize] \\
  =&  \frac{1}{\sqrt{n}}\Ee_{F_{n,i}^-} \left[\partial_i f\circ \normalize(S)\right]^T 0 + 
   \frac{1}{n}\text{Tr}\left[\Ee_{F_{n,i}^-} \left(\partial_i^2 f\circ\normalize(S)\right) \cdot \Sigma(\F_n)\right] + R_{n,i} 
    \end{aligned}
    $$
    where the precise form of the Taylor remainder $R_{n,i}$ depends
on realizing the laws $F_{n,i}$ and $F_{n,i}^-$ on the same probability space.
In order to not introduce new notation, we have avoided explicitly writing out this construction, directing
readers to \cite{chatterjee2005simple} for details. Nevertheless,
$$
\begin{aligned}
    |R_{n,i}| \leq c_1(p) [\lambda_3^{\Omega}(f) \cdot n^{-\frac{3}{2}}]
    \E_{F_{n,i}^-}\left[\exp\left(\Omega(\normalize(S)) + \frac{1}{\sqrt{n}}\Omega(\xi_{i,n}^0)\right) \|\xi^0_{i,n}\|_1^3 \right]
\end{aligned}
$$
where $\xi_{i,n}^0$ are centered version of $\xi_{i,n}$ 
and $c_1$ is some dimension dependent constant.

Let $C(\Omega)$ be the constant s.t $\Omega(\cdot) \leq C(\Omega) \|\cdot\|_1$, $C(\Omega)$ only depends on the dimension $p$.
    Thus, using the independence of the $\xi_{i,n}$'s,
    $$
    \begin{aligned}
    &\E_{F_{n,i}^-}\left[\exp\left(\Omega(\normalize(S)) + \frac{1}{\sqrt{n}}\Omega(\xi_{i,n}^0)\right)\|\xi^0_{i,n}\|_1^3 \right] \\
    \leq& 
    \Ee_{F_{n,i}^-} \left[\exp\left( C(\Omega) \|\normalize(S)\|_1\right)\right] \cdot \Ee_{\F_n} \left[\|\xi_{i,n}^0\|_1^3\exp\left(\frac{C(\Omega)\|\xi_{i,n}^0\|_1}{\sqrt{n}}\right)\right].
    \end{aligned}
    $$

Now we bound these two expectations. By the exponential moment condition \eqref{eq:exp_moment} and Lemma \ref{lem:mgf}, it is easy to
conclude the first term is bounded by
$$
\limsup_n \Ee_{\F_n} \left[\exp\left( C(\Omega) \|Z_n\|_1\right)\right] \leq c_2(p).
$$
The second expectation is bounded by $\gamma$, an upper bound on the third moment of $\xi_{i,n}^0$,
$$
\limsup_n \Ee_{\F_n} \left[\|\xi_{i,n}^0\|_1^3\exp\left(\frac{C(\Omega)\|\xi_{i,n}^0\|_1}{\sqrt{n}}\right)\right] \leq \gamma,
$$
Thus it is not hard to see
$$
|R_{n,i}| \leq c_1(p) c_2(p) \gamma \lambda_3^{\Omega}(f) n^{-\frac{3}{2}}.
$$

Notice the first and second order terms in $ \Ee_{F_{n,i}}[f\circ\normalize] - \Ee_{F_{n,i}^-}[f\circ\normalize]$ cancel with those in $\Ee_{F_{n,i-1}}[ f\circ\normalize] - \Ee_{F_{n,i}^-}[f\circ\normalize]$, and therefore we have,
    $$
    \big| \Ee_{\F_n}[ f] - \Ee_{\Phi_n}[ f] \big| \leq \sum_{i=1}^n (R_{n,i} + \tilde{R}_{n,i})    $$
    where $\tilde{R}_{n,i}$ is the remainder of $\Ee_{F_{n,i-1}}[ f\circ\normalize] - \Ee_{F_{n,i}^-}[f\circ\normalize]$.
With a similar argument
$$
|\tilde{R}_{n,i}| \leq c_1(p) c_2(p) \gamma \lambda_3^{\Omega}(f) n^{-\frac{3}{2}},
$$
and summing over $n$ terms, we have the conclusion of the lemma.
\end{proof}

Now we prove the main theorem, Theorem \ref{thm:weak:conv}.

\begin{proof}

First, notice that per Lemma \ref{lem:exact}, we have
$\Esubarg{\Phi_n^*}{g \circ P(T_n)} = \int_0^1 g(x) dx$. Using the selective likelihood ratio, it
is easy to see, 
$$
\Ee_{\F^*_n} \left[g(P(\stat_n))\right] = \Ee_{\F_n}\left[g(P(\stat_n)) \ell_{\F_n}(\stat_n)\right]
$$
The same equation holds for $\Phi_n = \Phi(\F_n)$, thus we have
\begin{equation}
    \label{eq:difference}
\begin{aligned}
&\left|
\Ee_{\F^*_n} \left[g(P(\stat_n))\right] - \Ee_{\Phi^*_n} \left[g(P(\stat_n))\right] \right| \leq  \\
& \quad \left| \Ee_{\F_n}\left[g(P(\stat_n)) \ell_{\Phi_n}(\stat_n)\right] -  \Ee_{\Phi_n}\left[g(P(\stat_n)) \ell_{\Phi_n}(\stat_n)\right] \right| + \\
& \quad \left| \Ee_{\F_n}\left[g(P(\stat_n)) \ell_{\F_n}(\stat_n)\right] -  \Ee_{\F_n}\left[g(P(\stat_n)) \ell_{\Phi_n}(\stat_n)\right] \right|  \\
\end{aligned}
\end{equation}

We need to bound both terms. Recall the notation $\bP_n$ and $\bL_{\F_n}$ for the normalized statistic $Z_n$. 
If we let $f = g(\bar{P}_n) \cdot \bar{\ell}_{\Phi_n}$, then per Lemma \ref{lem:ext:sourav}, we have
$$
|\Ee_{\F_n} [g(\bar{P}_n) \cdot \bL_{\Phi_n}] - \Ee_{\Phi_n} [g(\bar{P}_n) \cdot \bL_{\Phi_n}] | \leq 2  C(p) \cdot C_1 n^{-1/2},
$$
where we use the bound in condition \eqref{eq:derivatives}. Now we replace
$\ell_{\Phi_n}$ with $\ell_{\F_n}$ in the second term. With some algebra,
we can bound it by
$$
\Ee_{\F_n}\left[g(P(\stat_n)) \ell_{\F_n}(\stat_n)\right] 
\left| 1 - \frac{\Psubarg{\F_n \times \Q}{M \in \mQ^*(T_n, \omega)}}{\Psubarg{\Phi_n \times \Q}{M \in \mQ^*(T_n, \omega)}} \right|,
$$
which in turn is bounded by 
$$
C(g) \frac{\Pp_{(\F_n \times \Q)}[M \in \mQ^*] - \Pp_{(\Phi_n \times \Q)}[M \in \mQ^*]}{\Pp_{(\Phi_n \times \Q)}[M \in \mQ^*]}
\leq C(g) C_3 n^{-1/2},
$$
per condition \eqref{eq:rare} and $C(g)$ is a bound on $g$.
\end{proof}

\subsection{Proof of Lemma \ref{lem:exact}}
\begin{proof}
    Let $\phi_{\mu, \frac{1}{n}\Sigma}$ denote the density for $N(\mu, \frac{1}{n}\Sigma)$ 
    and $T = \frac{1}{\sigma_{\eta}^2}\Sigma\eta \cdot (\eta^T T) + V_{\eta}$, we see that $\Q(\eta^T T, V_{\eta})$ is $\modelweight(M; T)$ in \eqref{eq:selective_dbn}.
    Thus under the selective law $\F^*$, the distribution of $T$ has density proportional to 
    $$
    \phi_{\mu, \frac{1}{n} \Sigma}(T) \cdot \Q(\eta^T T, V_{\eta}).
    $$
    Since $\eta^T T \perp V_{\eta}$ under $\F$, we can factorize $\phi_{\mu, \frac{1}{n}\Sigma}(T)$ into the product of densities of $\eta^T T$ and $V_{\eta}$. Thus 
    conditioning on $V_{\eta}$, the density of $\eta^T T$ is proportional to 
    $$
    \exp\left[-\frac{n(\eta^T T - \eta^T \mu)^2}{2\sigma_{\eta}^2}\right] \cdot \Q(\eta^T T, V_{\eta}).
    $$
    Therefore, the pivot in \eqref{eq:pivot} is the survival function of $\eta^T T$ under $\F^*$ and is distributed as $\Unif(0,1)$. Moreover, we note the 
    distribution does not depend on the conditioned value of $V_{\eta}$, thus $P(T; \eta^T \mu, \Sigma)$ in \eqref{eq:pivot} is $\Unif(0,1)$.
\end{proof}

\bibliographystyle{agsm}
\bibliography{randomized_response}

\newpage
\appendix 
\section{Proof of Lemma \ref{lem:counter:eg}}
\begin{proof}
First we normalize the sample mean as $Z = \sqrt{n}(\bar{X}_n + 0.5)$ and rewrite the pivot as
$$
\tP(Z) = \frac{1 - \Phi(Z)}{1 - \Phi(2 + \frac{1}{2}\sqrt{n})}, \quad Z > 2 + \frac{1}{2}\sqrt{n},
$$
and $\Phi$ is the CDF of the standard normal distribution.
As $n \to \infty$, we can use Mills ratio to approximate the normal tail. Specifically, denote $b_n = \frac{1}{2}\sqrt{n} + 2$, 
\begin{equation}
\label{eq:pivot:approx}
\begin{aligned}
&\frac{1 - \Phi(Z)}{1 - \Phi(b_n)} \approx \frac{b_n}{Z} \exp\left[\frac{1}{2}(b_n+Z)(b_n -Z)\right]\\
=& \frac{b_n}{Z} \exp\left[-\frac{1}{2}(b_n-Z)^2\right] \exp\left[-b_n(Z - b_n)\right], \quad Z > b_n.
\end{aligned}
\end{equation}
We study the behavior of $-b_n(Z - b_n)$ for $Z > b_n$. By studying its distribution, we will
also see that $Z-b_n \overset{p}{\to} 0$, for $Z > b_n$, thus the term
$$
R_n = \frac{b_n}{Z} \exp\left[-\frac{1}{2}(b_n-Z)^2\right] \overset{p}{\to} 1, \text{ as } n \to \infty.
$$

Now we study the distribution of $b_n(Z - b_n)$ conditioning on $Z > b_n$. Since $\bar{X}_n$ is a translation
of a binomial distribution divided by $n$, we can rewrite $Z$ in terms of a Binomial distribution, which will
be useful for calculating the conditional distribution of $b_n(Z-b_n)$. Specifically, 
$$
Z = \frac{2S_n - n}{\sqrt{n}}, \quad S_n \sim \Bin(n,\frac{1}{2}).
$$
Thus for $t \geq 0$,  
$$
\begin{aligned}
&\Parg{b_n(Z - b_n) > t} = \Parg{b_n \left(\frac{2S_n - n}{\sqrt{n}} - b_n\right)} \\
=& \Parg{S_n > \sqrt{n} + \frac{3}{4}n + \frac{t\sqrt{n}}{2b_n}}
= \sum_{i=0}^{\frac{1}{4}n - \sqrt{n} - \frac{t\sqrt{n}}{2b_n}}\binom{n}{i} \cdot \frac{1}{2^n}
\end{aligned}
$$
To study the conditional distribution $\Parg{b_n(Z - b_n) > t\mid Z > b_n}$, we essentially
need to study the ratio of two partial sums of binomial coefficients. 

Note for any $n,m \in \integers_+$, $n > m$ we have
$$
\frac{\binom{n}{m-1}}{\binom{n}{m}} = \frac{m}{n-m+1}
$$ 
Noticing that $\frac{k}{n-k+1} \leq \frac{m}{n-m+1}$, for any $k \leq m$, thus
$$
\frac{\sum_{i=0}^{m-1} \binom{n}{i}}{\sum_{i=0}^m \binom{n}{i}} \leq \frac{m}{n-m+1}.
$$
Now let $m = \frac{1}{4}n - \sqrt{n}$, and use the above inequality $j = \frac{t\sqrt{n}}{2b_n}$ times, we have
$$
\frac{\sum_{i=0}^{m-j} \binom{n}{i}}{\sum_{i=0}^m \binom{n}{i}} \leq \left(\frac{m}{n-m+1}\right)^j.
$$
Therefore we have,
\begin{equation}
\label{eq:overshoot}
\begin{aligned}
&\Parg{b_n(Z - b_n) > t \mid Z > b_n} = \frac{\Parg{b_n(Z - b_n) > t}}{\Parg{b_n(Z-b_n) > 0}}\\
\leq& \left(\frac{\frac{1}{4}n - \sqrt{n}}{\frac{3}{4} n + \sqrt{n} + 1}\right)^{\frac{t}{1 + \frac{4}{\sqrt{n}}}} 
\rightarrow \exp[ -\log(3) t]
\end{aligned}
\end{equation}

We can draw two conclusions from \eqref{eq:overshoot}. First, conditional on $Z > b_n$, $Z - b_n \overset{p}{\rightarrow} 0$,
which implies the first term in the pivot approximation \eqref{eq:pivot:approx} $R_n \to 1$. Moreover, \eqref{eq:overshoot}
shows that the overshoot $b_n(Z-b_n)$ is not $\Exp(1)$ distributed in the limit. In fact, we can conclude its limit (if existed)
is strictly stochastically dominated by an $\Exp(1)$. Thus,
$$
\exp\left[- b_n(Z - b_n)\right] \not\rightarrow \Unif(0,1),
$$
and hence the pivot does not converge to $\Unif(0,1)$.
\end{proof}

\section{Proof of Lemma \ref{lem:logistic}}

\begin{proof}
We first prove that $T$ is in fact a linearizable statistic. Since $\bbeta_E$ is the
restricted MLE, we see that 
$$
\begin{aligned}
0 &= \frac{1}{n}X_E^T \left[y - \pi_E(\bbeta_E)\right]\\
&= \frac{1}{n}X_E^T[y - \pi_E(\beta_E^*)] + Q(\beta^*_E - \bbeta_E) + R_1,
\end{aligned}
$$
where $R_1 = (Q_E(\beta_E^*) - Q)(\beta^*_E - \bbeta_E) + \tilde{R}_1$, where $\tilde{R}_1 = o_p(n^{-1/2})$
is the residual from the Taylor's expansion at $\beta_E^*$. $R_1 = o_p(n^{-1/2})$ since
deviations $Q_E(\beta_E^*)$ from its asymptotic mean should be $O_p(n^{-1/2})$ and $\bbeta_E - \beta^*_E = o_p(1)$.

Thus, we can deduce
$$
\bbeta_E = \frac{1}{n}Q^{-1}X_E^T[y - \pi_E(\beta_E^*)] + \beta_E^* + Q^{-1}R_1.
$$
Similarly,
$$
\frac{1}{n} X_{-E}^T \left[y - \pi_E(X_E\bbeta_E)\right]
= \frac{1}{n}X_{-E}^T[y - \pi_E(\beta_E^*)] - \frac{1}{n}DX_E^T[y - \pi_E(\beta_E^*)] + o_p(n^{-1/2})
$$
Thus we can conclude that $T$ is a linearizable statistic with
$$
\xi_{i,n} = \begin{pmatrix}
QX_{i,E}^T (y_i - \pi(x_{i,E}\beta_E^*)) \\
X_{i,-E}^T (y_i - \pi(x_{i,E}\beta_E^*)) - DX_{i,E} (y_i - \pi(x_{i,E}\beta_E^*))
\end{pmatrix}.
$$
Now we rewrite the selection event in terms of $(T, \omega)$. Using the KKT conditions of \eqref{eq:randomized_logistic},
\begin{equation}
    \label{eq:KKT_rr}
\begin{aligned}
        &X^T(y - \pi_E(\hbeta_E)) = \sqrt{n}(\omega + \Lambda z) + \begin{pmatrix} 
            \hbeta_E \\
            0
        \end{pmatrix}\\
        &s_E\hbeta_E \geq 0, \quad \|u_{-E}\|_{\infty} < 1,
\end{aligned}
\end{equation}
where $z = \begin{pmatrix}
    s_E \\
    u_{-E}
\end{pmatrix}$, $s_E = \sign(\hbeta_E)$ and $u_{-E}$ is the subgradient for the inactive variables. 
Using a Taylor expansion on the $\hbeta_E$ as well, we see that
$$
\frac{1}{n} X^T[y - \pi_E(\hbeta_E)] = \frac{1}{n}
\begin{pmatrix}
0\\
X_{-E}^T[y - \pi_E(\bbeta_E)]
\end{pmatrix} + 
\begin{pmatrix}
Q \\
C
\end{pmatrix} (\bbeta_E - \hbeta_E) + o_p(n^{-1/2}).
$$

Plugging in the equalities in the KKT conditions, we will have,
$$
\begin{aligned}
\hbeta_E &= \bbeta_E - \frac{1}{\sqrt{n}} Q^{-1} (\omega_E + \Lambda_E z_E) + o_p(n^{-1/2})\\
\frac{1}{n} X_{-E}^T[y - \pi_E(\hbeta_E)] &= \frac{1}{n} X_{-E}^T[y - \pi_E(\bbeta_E)] + C(\bbeta_E - \hbeta_E) + o_p(n^{-1/2})\\
&= \frac{1}{n} X_{-E}^T[y - \pi_E(\bbeta_E)] + \frac{1}{\sqrt{n}}D(\omega_E + \Lambda_E z_E) + o_p(n^{-1/2})
\end{aligned}
$$
Using the inequalities in the KKT conditions, we have the selection event
is $\{A_M T + B_M \omega \leq b_M\}$ with $A_M$, $B_M$ and $b_M$ defined in the lemma.

\end{proof}

\section{Proofs related to Logistic noise}
Throughout the article, logistic noise has played an important role in all the examples.

The following lemma on the tail behavior of the logistic distribution is crucial to
all the proofs with added logistic noise.
Let $G$ be the CDF of $\Logistic(\kappa)$, with $\kappa$ being the scale parameter. $g$ is the PDF of $G$.
$$
G(w) = \frac{e^{\kappa w}}{1 + e^{\kappa w}}, \quad g(w) = \frac{\kappa e^{-\kappa |w|}}{\left(1+e^{-\kappa |w|}\right)^2}.
$$
\begin{lemma}
\label{lem:tail}
The following lower bounds hold,
\begin{equation}
\label{eq:logistic:lower}
\bar{G}(\kappa w) \geq \frac{1}{2} e^{-\kappa |w|}, \quad g(w) \geq \frac{1}{4} e^{-\kappa |w|}.
\end{equation}
For $k=0,1,2,3,\dots$:
\begin{equation}
\label{eq:logistic:upper}
\left|\frac{\partial ^k}{\partial w^k} g(w)\right| \leq \kappa^{k+1} \tC_k e^{-\kappa |w|}.
\end{equation}
where $\tC_k$'s are universal constants.
\end{lemma}
\begin{proof}
We can write
$$
g(w) = 
\begin{cases}
\kappa e^{-\kappa w} h_0(e^{-\kappa w}) & w > 0 \\
\kappa e^{\kappa w} h_0(e^{\kappa w}) & w \leq 0 \\
\end{cases}
$$
where $h_0(x) = (1+x)^{-2}$. 
For $j \geq 1$, define $h_j(x) = x \cdot h'_{j-1}(x)$. By induction, I claim that for each $j$, $h_j$ is rational such that the polynomial in the numerator is of order 2 less than the denominator, and the denominator polynomial is bounded below by 1. Hence, $h_j$'s are bounded on the interval $[0,1]$. Now, it is not hard to see that
$$
\frac{\partial^k }{\partial w^k} g(w) =
\begin{cases}
-(-\kappa)^{k+1} \sum_{j=0}^k c_{j,k} h_j(e^{-\kappa w}) e^{-\kappa w} & w > 0 \\
\kappa^{k+1} \sum_{j=0}^k c_{j,k} h_{j}(e^{\kappa w}) e^{\kappa w} & w \leq 0 \\
\end{cases}
$$
for universal $c_{j,k}$'s and $k=0,1,2,\dots$. 
\end{proof}

Now we state the following lemmas which are foundations of the proofs of various lemmas in the article. 

\begin{lemma}
\label{lem:mgf}
Assume $T_n$ is a decomposable statistic and $\xi_{i,n}$ has mean $0$, variance $\sigma^2$,
and centered exponential moments in a neighbourhood of zero, i.e
satisfies condition (\ref{eq:exp_moment}). Denote $Z_n = \sqrt{n}(T_n - \mu_n)$, then 
$$
\Earg{\exp\left( \kappa Z_n\right)} \rightarrow \exp\left(\frac{\kappa^2\sigma^2}{2}\right), \text{ for } \kappa > 0.
$$
\end{lemma}

\begin{lemma}
\label{lem:logistic:exp}
In Example \ref{example2}, if we normalize the sample mean $Z_n = \sqrt{n}(\bar{X}_n - \mu_n)$, we can
rewrite the selective likelihood ratio and the pivot as
$$
\bL_{\F_n}(Z_n) = \frac{\bar{G}(2 - Z_n - \sqrt{n}\mu_n)}{\Esubarg{\F_n}{\bar{G}(2 - Z_n - \sqrt{n}\mu_n)}},
$$
and
$$
\bP(Z_n) = \frac{\int_{Z_n}^{\infty} \bar{G}(2 - t - \sqrt{n}\mu_n) \exp(-t^2/2)\; dt}{\int_{-\infty}^{\infty} \bar{G}(2 - t - \sqrt{n}\mu_n) \exp(-t^2/2)\; dt}. 
$$

Then for any $\F_n$ with finite centered exponential moment in a neighbourhood of zero, we have for $k=0,1,2,3$
\begin{equation}
\label{eq:logistic:derivatives}
\frac{\partial^k}{\partial Z^k} \bL_{\F_n}(Z) \leq C_1 \exp[\kappa |Z|], \qquad \frac{\partial^k}{\partial Z^k} \bP(Z) \leq C_1.
\end{equation}
for some $C_1$ only depending on $\kappa$.
\end{lemma}

Proof of Lemma \ref{lem:mgf}.
\begin{proof}
Without loss of generality, we assume $T_n = \frac{1}{n}\sum_{i=1}^n \xi_{i,n}$.
Since $\F_n$ has centered exponential moments in a neighbourhood of zero, it is each to see
$$
\Earg{\exp\left( \kappa Z_n\right)} = \left[M\left(\frac{\kappa}{\sqrt{n}}\right)\right]^n 
$$
exists as long as $\frac{\kappa}{\sqrt{n}} < a$. $M(\cdot)$ is the moment generating
function of $\xi_{i,n} - \mu_n$, $M(t) = \Esubarg{\F_n}{\exp(t(\xi_{i,n} - \mu))}$. 
Therefore,
$$
\begin{aligned}
&\lim_{n\to \infty} n \log\left[M\left(\frac{\kappa}{\sqrt{n}}\right)\right] 
 \overset{t = \frac{1}{\sqrt{n}}}{=} \lim_{t \to 0}\frac{\log\left[M(\kappa t)\right]}{t^2} \\
=& \lim_{t \to 0} \frac{\kappa^2 M''(\kappa t)}{2 M(\kappa t)} = \frac{\kappa^2\sigma^2}{2}.
\end{aligned}
$$
To derive the equality, we used $M''(0) = \Vararg{\xi_{i,n}} = \sigma^2$ and $M'(0) = 0, ~ M(0) = 1$.
\end{proof}

Proof of Lemma \ref{lem:logistic:exp}.
\begin{proof}
Noticing the lower bound in \eqref{eq:logistic:lower}, we have
$$
\Earg{\bar{G}(2 - Z_n - \sqrt{n}\mu_n)} \geq \frac{1}{2}\Earg{e^{\left(-\kappa\abv{2-Z_n-\sqrt{n}\mu_n}\right)}}
\geq \frac{1}{2}\Earg{e^{\left(-\kappa\abv{2-Z_n}-\kappa\sqrt{n}\abv{\mu_n}\right)}}
$$
On the other hand, using the upper bounds in \eqref{eq:logistic:upper}, we have for $k=1,2,3$,
$$
\frac{\partial^k}{\partial Z^k} \bL_{\F_n}(Z) 
\leq 2 \frac{\kappa^k \tC_{k-1} e^{-\kappa \abv{2-Z-\sqrt{n}\mu_n}}}{\Earg{e^{\left(-\kappa\abv{2-Z}-\kappa\sqrt{n}\abv{\mu_n}\right)}}}
\leq 2 \frac{\kappa^k \tC_{k-1} e^{\kappa \abv{2-Z}}}{\Earg{e^{\left(-\kappa\abv{2-Z}\right)}}}
$$
Since $x^{-1}$ is convex on the positive axis, it is hard to see
$$
\frac{1}{\Earg{e^{\left(-\kappa\abv{2-Z}\right)}}} \leq \Earg{e^{\left(\kappa\abv{2-Z}\right)}} \leq e^{2\kappa} \Earg{e^{\kappa Z} + e^{-\kappa Z}}. 
$$
Thus using Lemma \ref{lem:mgf}, we know $\Earg{e^{\pm \kappa Z}} \to \exp(\kappa^2/2)$. Thus, we conclude
\begin{equation}
\label{eq:logistic:123}
\sup_n \frac{\partial^k}{\partial Z^k} \bL_{\F_n}(Z)  \leq C_1 \exp[\kappa \abv{Z}], \quad k=1,2,3.
\end{equation}
To verify the above inequality for $k=0$. Notice that for $\mu_n \geq 0$, $\bar{G}(2 - Z - \sqrt{n}\mu_n) \leq \bar{G}(2 - Z)$.
Thus the denominator of $\bL_{\F_n}(Z)$ is bounded below using the argument above. For $\mu < 0$,
$$
\bar{G}(2 - Z - \sqrt{n}\mu_n) \leq \exp(-\kappa(2-Z -\sqrt{n}\mu_n) \leq \exp(-\kappa\sqrt{n}\abv{\mu_n} + \kappa \abv{2-Z}).
$$ 
The term $\exp(-\kappa\sqrt{n}\abv{\mu_n})$ cancels with the one in the denominator, thus \eqref{eq:logistic:123} holds for $k=0$ as well. 

Analogously, similar bounds can be derived for the derivatives of $\bP(Z)$ as well, thus we have the conclusion of the lemma.
\end{proof}

\subsection{Proof of Lemma \ref{lem:simple:consistent} and Lemma \ref{lem:file:drawer}}.

The proof of Lemma \ref{lem:simple:consistent} is a simple application of Lemma \ref{lem:mgf} and
Lemma \ref{lem:logistic:exp}.
\begin{proof}
By law of large numbers, we know that $\bar{X}_n$ is consistent for $\mu$ unselectively. Thus,
using the result by Lemma \ref{lem:transfer}, we only need to verify that the selective likelihood
is integrable in $L^q$. For simplicity, we take $q=2$. 

First notice from \eqref{eq:logistic:derivatives} that the selective likelihood ratio is bounded
by a multiple of $\exp[\kappa |Z|]$. Then by Lemma \ref{lem:mgf},
$$
\limsup_{n} \Esubarg{\F_n}{\ell_{\F_n}(\bar{X})^2} \leq 2C_1^2\exp(2\kappa^2).
$$
\end{proof}

The proof of Lemma \ref{lem:file:drawer} uses results in Lemma \ref{lem:logistic:exp} and
Lemma \ref{lem:ext:sourav}
\begin{proof}
It follows simply from \eqref{eq:logistic:derivatives} that condition \eqref{eq:derivatives} are
satisfied with the norm function $\Omega$ simply being the absolute value function. Therefore,
we only need to verify \eqref{eq:rare}. Note for $\mu_n = \mu < 0$
$$
\begin{aligned}
&\frac{\Psubarg{\F_n \times \Q}{\sqrt{n}\bX_n + \omega > 2}- \Psubarg{\Phi \times \Q}{\sqrt{n}\bX_n + \omega > 2}}{\Psubarg{\Phi \times \Q}{\sqrt{n}\bX_n + \omega > 2}}\\
=&\frac{\Esubarg{\F_n}{\bG(2 - Z - \sqrt{n}\mu)}- \Esubarg{\Phi}{\bG(2 - Z - \sqrt{n}\mu)}}{\Esubarg{\Phi}{\bar{G}(2 - Z - \sqrt{n}\mu)}} \\
\leq& 2\frac{\exp(-\kappa\sqrt{n}\mu)\Esubarg{\F_n}{\bG(2 - Z - \sqrt{n}\mu)}- \Esubarg{\Phi}{\bG(2 - Z - \sqrt{n}\mu)}}{\Esubarg{\Phi}{\exp(-\kappa|2 - Z|)}}\\
\leq& 2\Esubarg{\Phi}{\exp(2\kappa + 2|Z|)} \cdot C \exp(-\kappa\sqrt{n}\mu) \exp(\kappa\sqrt{n}\mu) n^{-1/2}\\
\leq& 2C\exp(\kappa^2) n^{-1/2}, 
\end{aligned}
$$
The second to last inequality uses Lemma \ref{lem:ext:sourav} and the fact that
$\lambda_3^{\Omega}(\bar{G}) \leq \exp(\kappa\sqrt{n}\mu)$. For $\mu_n = \mu \geq 0$, the denominator $\Psubarg{\Phi \times \Q}{\sqrt{n}\bX_n + \omega > 2}$ 
is bounded below, and $\bar{G}$ has bounded derivatives. Therefore, a simple application of the Berry-Esseen Theorem will suffice.
\end{proof}

\section{Proofs related to affine selection regions}

\subsection{Proof of Lemma \ref{lem:asymptotic}}

The quantity that appears in both the pivot and the selective likelihood ratio is
$$
\Q(z;\Delta) = \Pp(A(z+\Delta) + \omega \in {\selectionevent}) = \int_{{\selectionevent} - A(z+\Delta)} \; G(dw),
$$
where $\omega \sim G$. The associated selective likelihood in terms of $z$ is 
$$
\ell_{\F}(z;\Delta) = \frac{\Q(z;\Delta)}{\int_{\real^q} \Q(t;\Delta) \F(dt)}.
$$
We first rewrite the pivot in terms of $U$.
\begin{equation}
    \label{eq:pivot_rr}
    P(z;\Delta) = \dfrac{\int_{\eta^T z}^{\infty}\Q(t; Lz,\Delta)\exp(-t^2/2\sigma_{\eta}^2) \; dt}{\int_{-\infty}^{\infty}\Q(t; Lz,\Delta)\exp(-t^2/2\sigma_{\eta}^2) \; dt},
\end{equation}
where we use the slight abuse of notation for the one dimensional function $\Q(t;Lz,\Delta)$
$$
\begin{aligned}
    L &= I - \frac{1}{\sigma^2_{\eta}} \Sigma \eta\eta^T \\
\Q(t; Lu,\Delta) &= \Pp(t\cdot\frac{1}{\sigma_{\eta}^2} A\Sigma\eta + ALu + A\Delta + w \in K) \\
&= \int_{K-t\cdot\frac{1}{\sigma_{\eta}^2} A\Sigma\eta - ALu - A\Delta} G(dw).
\end{aligned}
$$

We first establish a lower bound on $\Esubarg{\F}{\Q(Z;\Delta)} = \Psubarg{\F_n \times \Q}{A (Z+\Delta) + \omega \in K}$
under the local alternatives.

\begin{lemma}
\label{lem:rare:lower}
If we assume the lower bound condition, then under
the local alternatives with radius $B$, i.e. $d_h(0, K-A\Delta) \leq B$, we have
$$
\int_{\real^p} \Q(u;\Delta) \F(du) \geq C^-C(\Phi, h)\cdot e^{-B},
$$
where $C(\Phi, h)$ is a constant only depending on the normal distribution $\Phi = N(0,\Sigma)$ and
the norm $h$ in the local alternatives condition.
\end{lemma}

\begin{proof}
We first see that the lower bound condition gives the following lower bound.
\begin{equation}
    \label{eq:lower}
    \Q(u;\Delta) = \int_{K-A(u+\Delta)} G(dw) \geq C^- \exp\left[-\inf_{w \in K-A(u+\Delta)} h(w)\right].
\end{equation}

Consider
$$
\begin{aligned}
\int_{\real^p} \Q(u;\Delta) \F(du) 
& \geq C^-\int_{\real^p} \exp\left(-\inf_{w \in K- A(\Delta+u)} h(w) \right) \; \F(du) \\
& = C^- \int_{\real^p} \exp\left(-\inf_{w \in K-A\Delta} h(w-Au)\right) \; \F(du)\\
& \geq C^-\int_{\real^p} \exp\left(-\inf_{w \in K-A\Delta} h(w) + h(-Au)\right) \; \F(du)\\
& = C^- \cdot \exp\left(-\inf_{w \in K-A\Delta} h(w)\right) \int_{\real^p}  e^{-h(Au)} \; \F(du) \\
&= C^-\int_{\real^p}  e^{-h(Au)} \; \F(du) \cdot e^{-B}. 
\end{aligned}
$$

Finally, since the $\exp(-h(Au))$ has uniformly bounded derivatives up to the third order,
we have
$$
\int_{\real^p} \exp(-h(Au)) \F(du) \to \int_{\real^p} \exp(-h(Au)) \Phi(du)
$$
as $Z \to N(0, \Sigma)$ in distribution. Let $C(\Phi,h) = \int_{\real^p} \exp(-h(Au)) \Phi(du)$,
and we will have the conclusion of the lemma.
\end{proof}

The following lemmas establish the bounds on the derivatives for the likelihood function $\ell_{\F}$ and the pivot $P(u;\Delta)$.
Lemma \ref{lem:asymptotic} is easily obtained using Lemma \ref{lem:LR} and Lemma \ref{lem:derivative} below. 

\begin{lemma}
\label{lem:LR}
Suppose the smoothness and the lower bound conditions are satisfied,  
then for local alternatives with radius $B$,
\begin{equation}
    \frac{\partial^{\alpha}}{\partial z^{\alpha}}\ell_{\F}(z;\Delta) \leq C^*(B,\Phi,h), \quad \Phi = N(0,\Sigma).
\end{equation}
\end{lemma}

\begin{proof}
The smoothness condition implies the following upper bound. For a multi-index $\alpha$, we have
$$
\begin{aligned}
\frac{\partial^{\alpha}}{\partial z^{\alpha}} \Q(z;\Delta)
&= \frac{\partial^{\alpha}}{\partial z^{\alpha}} \int_{\selectionevent-A(\Delta+z)} g(w) dw \\
&=  \int_{\selectionevent-A\Delta} \frac{\partial^{\alpha}}{\partial z^{\alpha}} g(w-Az) dw. \\
\end{aligned}
$$
Therefore, from the smoothness condition,
\begin{equation}
\label{eq:upper}
\begin{aligned}
    \bigg\|\frac{\partial^{\alpha}}{\partial z^{\alpha}} \Q(z;\Delta) \bigg\| 
    &\leq C(A) C_{\alpha} \overset{\mrm{def}}{=} C^+_{\alpha}(A).
\end{aligned}
\end{equation}
This combined with Lemma \ref{lem:rare:lower} gives the conclusion of the lemma.
\end{proof}

Next, we derive the exponential bounds on the derivatives of the pivot $P(z;\Delta)$ with respect to $z$.

\begin{lemma}
    \label{lem:derivative}
    Assuming the conditions of Lemma \ref{lem:asymptotic}, for a multi-index $\alpha$ up to the order of $3$, 
    $$
    \bigg\|\frac{\partial^{\alpha}}{\partial z^{\alpha}} P(z;\Delta)\bigg\| \leq C^*_{\alpha} \cdot e^{\alpha \cdot \text{Lip}(h)\|ALz\|_2},
    $$
    where the norm on the left is the element-wise maximum and $C$ is independent of $(z,\Delta)$ and 
    $\text{Lip}(h)$ is the Lipschitz constant of $h$ with respect to $\ell_2$ norm.
\end{lemma}

\begin{proof}
    To get a lower bound on the denominator, note \eqref{eq:lower}
    $$
    \begin{aligned}
        \Q(t;Lz,\Delta) &\geq C^- \exp\left[- \inf_{w \in K - ALz - A\Delta} h(w-t\cdot \frac{1}{\sigma_{\eta}^2} A\Sigma\eta)\right] \\ 
        &\geq C^- \exp\left[- \inf_{w \in K - ALz - A\Delta} h(w) - |t|h(A\Sigma\eta/\sigma^2_{\eta})\right].
    \end{aligned}
    $$
    Therefore, the denominator will be lower bounded by
    $$
    \begin{aligned}
    &\frac{1}{\sqrt{2\pi}}\int_{-\infty}^{\infty}\Q(t;Lz,\Delta) \exp(-t^2/2\sigma^2_{\eta}) dt \\
    \geq &C^- \exp\left[h(A\Sigma\eta/\sigma^2_{\eta})^2\sigma^2_{\eta}/2\right] \cdot \exp\left[- \inf_{w \in K - ALz - A\Delta} h(w)\right]\\
    \geq &C^- \exp\left[h(A\Sigma\eta/\sigma^2_{\eta})^2\sigma^2_{\eta}/2\right]e^{-B} \cdot e^{-h(ALz)}. 
    \end{aligned}
    $$
    On the other hand, the upper bound \eqref{eq:upper} ensures,
    $$
    \frac{1}{\sqrt{2\pi}}\int_{\real}\bigg\|\frac{\partial^{\alpha}}{\partial z^{\alpha}}\Q(t;Lz,\Delta) \bigg\| \exp(-t^2/2\sigma_{\eta}^2) dt \leq C_{\alpha}^+(A). 
    $$
    Note the derivatives of the pivot will be a polynomial in terms of the form,
    $$
    \frac{\int_{\eta^T z}^{\infty} \partial_{\alpha}\Q(t;Lz,\Delta) \exp(-t^2/2\sigma_{\eta}^2) dt}{\int_{-\infty}^{\infty}\Q(t;Lz,\Delta) \exp(-t^2/2\sigma^2_{\eta}) dt}
    $$
    and therefore, it is easy to get the conclusion of the lemma.
\end{proof}

\subsection{Proof of Lemma \ref{lem:berry}}

Using Lemma \ref{lem:rare:lower} and the following lemma, we can easily prove Lemma \ref{lem:berry}.

\begin{lemma}
\label{lem:berry:ext}
Let $Z_n = \sqrt{n}(T_n - \mu_n) \in \real^p$, and $\F_n$ has finite third moments $\gamma$. Moreover,
suppose the randomization noise $\omega \in \Q$, a probability measure on $\real^d$. Then for any sequence of sets
$(U_n)_{n \geq 1}, ~ U_n \subseteq \real^p \times \real^d$, we have
$$
\abv{\Psubarg{\F_n \times \Q}{(Z_n, \omega) \in U_n} - \Psubarg{\Phi_n \times \Q}{(Z_n, \omega) \in U_n}} \leq C_3 \gamma n^{-\frac{1}{2}},
$$
where $\Phi_n = N(\mu(\F_n), \Sigma(\F_n))$ and $C_3$ is a constant depending only on $p$.
\end{lemma}

Proof of Lemma \ref{lem:berry:ext} uses the well known results of Berry-Esseen Theorem. A multivariate extension
can be found in \cite{multivariate_clt}.
\begin{proof}
For each $\omega$, we denote
$$
U_n(\omega) = \{Z \in \real^p: (Z, \omega) \in U_n\} \subseteq \real^p.
$$
Thus the difference in the two probabilities is
$$
\begin{aligned}
&\abv{\Psubarg{\F_n \times \Q}{(Z,\omega) \in U_n} - \Psubarg{\Phi_n \times \Q}{(Z,\omega) \in U_n}}\\
\leq & \Esubarg{\Q}{\abv{\Psubarg{\F_n}{Z \in U_n(\omega)} - \Psubarg{\Phi}{Z \in U_n(\omega)}}} \\
\leq & \sup_{U \in \real^p} \abv{\F_n(U) - \Phi_n(U)} < C_3\gamma n^{-1/2},
\end{aligned}
$$
where $C_3$ only depends on the dimension $p$. The last inequality is a direct application of equation
(1.5) in \cite{multivariate_clt}.
\end{proof}

\end{document}